\documentclass[12pt]{amsart}
\usepackage{amssymb}
\usepackage{eucal}
\usepackage[bookmarks=false]{hyperref}
\usepackage{bm}
\usepackage[small,nohug,heads=vee]{diagrams}
\diagramstyle[labelstyle=\scriptstyle]

\allowdisplaybreaks

\newtheorem{theorem}{Theorem}[section]
\newtheorem{lemma}[theorem]{Lemma}
\newtheorem{proposition}[theorem]{Proposition}
\newtheorem{corollary}[theorem]{Corollary}

\newtheorem*{theorem:Period}{Theorem~\ref{T:Period}}

\theoremstyle{definition}

\newtheorem{example}[theorem]{Example}

\theoremstyle{remark}
\newtheorem{remark}[theorem]{Remark}

\newcommand{\C}{\ensuremath \mathbb{C}}

\newcommand{\BB}{\mathbb{B}}

\newcommand{\F}{\ensuremath \mathbb{F}}

\newcommand{\TT}{\mathbb{T}}

\newcommand{\bA}{\mathbf{A}}

\newcommand{\bK}{\mathbf{K}}
\newcommand{\ba}{\mathbf{a}}
\newcommand{\bb}{\mathbf{b}}
\newcommand{\bc}{\mathbf{c}}
\newcommand{\bd}{\mathbf{d}}
\newcommand{\bu}{\mathbf{u}}

\newcommand{\bg}{\mathbf{g}}
\newcommand{\bh}{\mathbf{h}}

\newcommand{\bz}{\mathbf{z}}

\newcommand{\cE}{\mathcal{E}}

\newcommand{\cL}{\mathcal{L}}
\newcommand{\cO}{\mathcal{O}}

\newcommand{\cU}{\mathcal{U}}

\newcommand{\cM}{\mathcal{M}}

\newcommand{\isom}{\ensuremath \cong}
\newcommand{\inv}{\ensuremath ^{-1}}

\newcommand{\on}{\ensuremath ^{\otimes n}}

\DeclareMathOperator{\Mat}{Mat}
\DeclareMathOperator{\divisor}{div}
\DeclareMathOperator{\End}{End}
\DeclareMathOperator{\Exp}{Exp}
\DeclareMathOperator{\Log}{Log}

\DeclareMathOperator{\Fr}{Fr}

\DeclareMathOperator{\ord}{ord}

\DeclareMathOperator{\Res}{Res}
\DeclareMathOperator{\RES}{RES}
\DeclareMathOperator{\sgn}{sgn}
\DeclareMathOperator{\Span}{Span}
\DeclareMathOperator{\Spec}{Spec}

\newcommand{\twist}{^{(1)}}

\newcommand{\twistinv}{^{(-1)}}
\newcommand{\twisti}{^{(i)}}
\newcommand{\twistk}[1]{^{(#1)}}

\newcommand{\power}[2]{{#1 [[ #2 ]]}}

\newcommand{\tpi}{\widetilde{\pi}}

\newcommand{\tsgn}{\widetilde{\sgn}}

\def\XXint#1#2#3{{\setbox0=\hbox{$#1{#2#3}{\int}$ }
\vcenter{\hbox{$#2#3$ }}\kern-.6\wd0}}

\begin{document}

\title[Tensor Powers and Periods]{Tensor powers of rank 1 Drinfeld modules and Periods}

\author{Nathan Green}
\address{Department of Mathematics, Texas A{\&}M University, College Station,
TX 77843, USA}
\email{ngreen@math.tamu.edu}

\subjclass[2010]{}

\date{}

\begin{abstract}
We study tensor powers of rank 1 sign-normalized Drinfeld $\mathbf{A}$-modules, where $\mathbf{A}$ is the coordinate ring of an elliptic curve over a finite field.  Using the theory of $\mathbf{A}$-motives, we find explicit formulas for the $\mathbf{A}$-action of these modules.  Then, by developing the theory of vector-valued Anderson generating functions, we give formulas for the period lattice of the associated exponential function.
\end{abstract}

\keywords{Tensor powers of Drinfeld modules, shtuka functions, reciprocal sums,  periods, motives}
\date{\today}
\maketitle

\tableofcontents

\section{Introduction} \label{S:Intro}
The Carlitz module and its tensor powers are well understood.  We have explicit formulas for multiplication maps of both the Carlitz module and for its tensor powers (see \cite{Carlitz35} for the Carlitz module and \cite[\S 3]{PLogAlg} for tensor powers of the Carlitz module).  Further, we have a nice product formula for $\tpi$, the Carlitz period, and a formula for the bottom coordinate of the fundamental period associated with tensor powers of the Carlitz module (see \cite[\S 2.5]{AndThak90}).

As a generalization for the Carlitz module, Drinfeld introduced the notion of Drinfeld modules (see also \cite{Goss}, \cite{Hayes92} or \cite{Thakur} for a thorough account of Drinfeld modules).  Since their introduction, many researchers have worked to developed an explicit theory for Drinfeld modules which parallels that for the Carlitz module, notably Anderson in \cite{And94} and \cite{And96}, Thakur in \cite{Thakur92} and \cite{Thakur93}, Dummit and Hayes in \cite{DummitHayes94}, and Hayes in \cite{Hayes79}.  To discuss the results of the present paper, we first recall a few basic facts about rank 1 sign-normalized Drinfeld modules over rings $\bA$, where $\bA$ is the affine coordinate ring of an elliptic curve $E/\F_q$ (see \S \ref{S:Amotives} for a more thorough review of Drinfeld modules).  Define $\bA = \F_q[t,y]$, where $t$ and $y$ are related via a cubic Weierstrass equation for $E$.  Also define an isomorphic copy of $\bA$, which we denote $A = \F_q[\theta,\eta]$, where $\theta$ and $\eta$ satisfy the same cubic Weierstrass equation as $t$ and $y$.  Let $K$ be the fraction field of $A$, let $K_\infty$ be the completion of $K$ at the infinite place, let $\C_\infty$ be the completion of an algebraic closure of $K_\infty$.  Let $H$ be the Hilbert class field of $K$, which can be taken to be a subfield of $K_\infty$.  A rank 1 sign-normalized Drinfeld module is an $\bA$-module homomorphism
\[\rho:\bA\to L[\tau]\]
satisfying certain naturally defined conditions, where $L\subset \C_\infty$ is some algebraically closed field containing $H$ and $L[\tau]$ is the ring of twisted polynomials in the $q$th power Frobenius endomorphism~$\tau$ (see \S \ref{S:Amotives} for definitions).   Associated to this Drinfeld module there is a point $V\in E(H)$ called the Drinfeld divisor, satisfying the equation on $E$
\[V\twist - V + \Xi = \infty,\]
where $\Xi = (\theta,\eta)\in K(E)$ and $V^{(1)}$ is the image of $V$ under the $q$th power Frobenius isogengy.  We specify that $V$ be in the formal group of $E$ at the infinite place of $K$, so that $V$ is uniquely determined by the above equation.  We define the shtuka function $f\in H(t,y)$ associated to $\rho$ to have
\[\divisor(f) = (V\twist) - (V) + (\Xi) - (\infty),\]
and require that the sign of $f$ equals 1 so that $f$ is uniquely determined (see \S \ref{S:Background} for the definition of sign).

Generalizing the Carlitz module further, Anderson introduced the notion of $n$th tensor powers of Drinfeld modules in \cite{And86}, which provide $n$-dimensional analogues of (1-dimensional) Drinfeld modules.  Then, in the remarkable paper \cite{AndThak90}, Anderson and Thakur develop much of the explicit theory for the arithmetic of the $n$th tensor power of the Carlitz module, including the aforementioned formula for the bottom coordinate of the fundamental period of the exponential function.  More recently, in \cite{PLogAlg} Papanikolas uses hyper derivatives to give extremely explicit formulas for multiplication maps and the fundamental period of tensor powers of the Carlitz module, along with with remarkable log-algebraicity theorems.   Both Anderson and Thakur's and Papanikolas's techniques allow them to connect the logarithm function to function field zeta values.

The goal of the present paper is to lay the groundwork for connecting the logarithm function associated to tensor powers of rank 1 Drinfeld modules with zeta values.  This connection is achieved in a concurrent paper by the author (see \cite{GreenZeta17}), which was written separately due to length considerations.  This notion of using Drinfeld modules to study $L$-functions, zeta functions, and their special values over functions fields has been pursued vigorously in the last few years and has born much fruit (see \cite{AnglesNgoDacRibeiro16a}-\cite{AnglesSimon}, \cite{Goss13}, \cite{LutesThesis}, \cite{LP13}, \cite{Pellarin08}-\cite{Perkins14b}, \cite{Taelman12}).

The main focus of this paper is the study of tensor powers of rank 1 sign-normalized Drinfeld modules over the affine coordinate ring of an elliptic curve.  These modules provide a further generalization of the Carlitz module and are an example of Anderson $\bA$-modules.  An $n$-dimensional Anderson $\bA$-module is an $\bA$-module homomorphism
\[\rho:\bA \to \Mat_n(L)[\tau]\]
satisfying certain naturally defined conditions, where $\Mat_n(L)[\tau]$ is the ring of twisted polynomials in the $q$th power Frobenius endomorphism~$\tau$, which extends to matrices entry-wise (see \S \ref{S:Amodules} for the full definition of Anderson $\bA$-modules).  The main theorems of this paper give formulas for the $\bA$-action of tensor powers of rank 1 sign-normalized Drinfeld modules, as well as for the fundamental period of the exponential function associated to this module.  This generalizes both the work of Papanikolas and the author on Drinfeld modules in \cite{GP16} as well as that of Anderson and Thakur on tensor powers of the Carlitz module in \cite{AndThak90}.  One of the main new aspects of this work, which distinguishes it from that of Anderson and Thakur, is that we prove many of our results in a vector-valued setting.  In particular, we define and study vector-valued Anderson generating functions (see \eqref{Eudef}), and define new operators which act on these vector-valued functions (see \S \ref{S:Operators})

After setting out the notation and background in \S \ref{S:Background}, in \S \ref{S:Amotives} we begin by defining $\bA$-motives and dual $\bA$-motives, which are tensor powers of 1-dimensional motives.  We realize these $\bA$-motives and dual $\bA$-motives as spaces of functions
\[M = \Gamma (U,\cO_E(nV)),\quad N = \Gamma (U,\cO_E(-nV\twist)),\]
respectively, where $U = \Spec L[t,y]$ is the affine curve $(L \times_{\F_q} E) \setminus \{ \infty \}$.  The spaces $M$ and $N$ are generated as a free $L[\tau]$-module and a free $L[\sigma]$-module by the sets of naturally defined functions
\begin{equation}\label{basisfunctions}
\{g_1,\dots, g_n\},\quad \{h_1,\dots,h_n\},
\end{equation}
respectively, for $g_i,h_i\in L(t,y)$ (see \eqref{gidivisor} and \eqref{hidivisor} for specific definitions).  The functions $g_i$ and $h_i$ appear repeatedly throughout our paper, and one can think of them as a generalization of the shtuka function to the $n$-dimensional setting.

To ease notation throughout the paper, for a fixed dimension $n$, we define
\begin{equation}\label{Nidef}
N_i\in \Mat_n(\F_q)
\end{equation}
for an integer $i\geq 1$ to be the matrix with $1$'s along the $i$th super-diagonal and $0$'s elsewhere and define $N_i$ for $i\leq -1$ to be the matrix with $1$'s along the $i$th sub-diagonal and $0$'s elsewhere.  We also define $E_1$ to be the matrix with a single $1$ in the lower left corner and zeros elsewhere and in general define $E_i$ to be $N_{i-n}$.  We also define $N_{i}(\alpha_1,\dots,\alpha_{n-i})$ to be the matrix with the entries $\alpha_1,\alpha_2, \dots, \alpha_{n-i}$ along the $i$th super diagonal and similarly for $N_{i-n}(\alpha_1,\dots,\alpha_{n-i})$ and $E_i(\alpha_1,\dots, \alpha_i)$.

Using $M$ and $N$, in \S \ref{S:Amodules} we define an Anderson $\bA$-motive $\rho\on$, which is the $n$th tensor power of a (1-dimensional) rank 1 sign-normalized Drinfeld module, and analyze the structure of $\rho\on_t$ and $\rho\on_y$.  We find that
\begin{equation}\label{rhotdef1}
\rho_t\on = (\theta I + N_1 (a_1,\dots,a_{n-1}) + N_2) + (a_nE_1+E_2)\tau,
\end{equation}
where $a_i$ are naturally defined constants in $H$ (see \eqref{taction} and Corollary \eqref{C:aivalue}), and that $\rho_y\on$ is defined similarly (see \eqref{yaction}).  By way of comparison, recall that for the $n$th tensor power of the Carlitz module (see Example \ref{Carlitztensorpowers}), we can write
\[C_t\on = (\theta I + N_1) + E_1\tau.\]
We denote the exponential and logarithm functions associated to $\rho\on$ as
\[\Exp\on_\rho(\bz) = \sum_{i=0}^\infty Q_i \bz\twisti, \quad \Log\on_\rho(\bz) = \sum_{i=0}^\infty P_i \bz\twisti,\]
and denote the period lattice of $\Exp_\rho\on$ as $\Lambda_\rho\on$.

In \S \ref{S:Operators} we define an $\bA$-module of rigid analytic functions $\Omega_0$ which vanish under the operator $\tau - f^n$, where $\tau$ acts by twisting (see \eqref{twistdef} for a definition of twisting).  We then proceed to define an $n$-dimensional ``vector version" of the operator $\tau - f^n$ which we denote
\[G - E_1\tau \in \Mat_n(H(t,y))[\tau],\]
which acts on vectors of rigid analytic functions, and in Lemma \ref{L:GE1andOmega} we solidify the connection between these two operators.  These vector operators allow us in \S \ref{S:Period} to connect the fundamental period $\Pi_n$ of $\Exp_\rho\on$ with the vector space $\Omega_0$ and obtain formulas for $\Pi_n$.  To state the main theorem on periods, we begin by recalling the function 
\[
  \omega_\rho = \xi^{1/(q-1)} \prod_{i=0}^\infty \frac{\xi^{q^i}}{f^{(i)}},
\]
from \cite[\S 4]{GP16}, where $f^{(i)}$ is the $i$th twist of $f$ and $\xi = -(m + \beta/\alpha)$ (see \eqref{fdef} for the definition of $m$).  We also define vector valued Anderson generating functions,
\[E_{\bu}^{\otimes n}(t) =  \sum_{i=0}^\infty \Exp_\rho\on\left (d[\theta]^{-i-1} \bu\right ) t^i \in \TT^n,\]
where $\bu \in \C_\infty$ and $\TT$ is a Tate algebra (see \eqref{Tatealgs} for the definition of $\TT$), and prove several properties about them.  We relate the function $\omega_\rho^n$ to $E_{\bu}^{\otimes n}$ using the vector operator $G-E_1\tau$ from \S \ref{S:Operators}.
Using these techniques, we get the following information about the period lattice.
\begin{theorem:Period}
If we denote
\[\Pi_n = -\left (\begin{matrix}
\Res_{\Xi}(\omega_\rho^n g_1\lambda)  \\
\vdots  \\
\Res_{\Xi}(\omega_\rho^n g_n\lambda) \\
\end{matrix}\right ),\]
where $g_i$ are the functions from \eqref{basisfunctions} and $\lambda$ is the invariant differential on $E$, then the structure of the period lattice of $\Exp_\rho\on$ is given by
\[\Lambda_\rho\on = \{d[a]\Pi_n\mid a\in \bA\},\]
where $d[a]$ is the constant term of $\rho_a\on$.  Further, if $\pi_\rho$ is a fundamental period of the (1-dimensional) Drinfeld exponential function, then the last coordinate of $\Pi_n \in \C_\infty^n$ is
\[\frac{g_1(\Xi)}{a_1a_2\dots a_{n-1}}\cdot\pi_\rho^n,\]
where the constants $a_i$ are the same as in \eqref{rhotdef1}.
\end{theorem:Period}

In the concurrent paper by the author \cite{GreenZeta17} we build upon the techniques of this paper to establish formulas for the coefficients of $\Exp_\rho\on$ and for $\Log_\rho\on$ and relate values of the logarithm function to special values of the zeta function
\[\zeta(n) = \sum_{\substack{a\in A \\ a\text{ monic}}} \frac{1}{a^n}.\]

\section{Background and Notation}\label{S:Background}
We require much of the same notation as \cite[\S 2]{GP16} and we use similar exposition in this section.  Let $p$ be a prime and $q=p^r$ for some integer $r>0$ and let $\F_q$ be the field with $q$ elements.  Define the elliptic curve $E$ over $\F_q$ with Weierstrass equation
\begin{equation}\label{ecequation}
E: y^2 + c_1 ty + c_3y = t^3 + c_2 t^2 + c_4t + c_6,\quad c_i\in \F_q,
\end{equation}
with the point at infinity designated as $\infty$.  Let $\bA = \F_q[t,y]$ be the affine coordinate ring of $E$, the functions on $E$ regular away from $\infty$, and let $\bK=\F_q(t,y)$ be its fraction field.  Let
\begin{equation}\label{invdiff}
\lambda = \frac{dt}{2y + c_1t + c_3}
\end{equation}
be the invariant differential on $E$.  Also define isomorphic copies of $\bA$ and $\bK$ with an independent set of variables $\theta$ and $\eta$, which also satisfy \eqref{ecequation}, which we label
\[A = \F_q[\theta,\eta],\quad \text{and}\quad K = \F_q(\theta,\eta).\]
Define the canonical isomorphisms
\begin{equation}\label{canoniso}
\iota:\bK\to K, \quad \chi:K\to \bK
\end{equation}
such that $\iota(t)=\theta$ and $\iota(y) = \eta$ and so on.  For convenience, for $x\in K$ we will sometimes refer to $\chi(x) = \overline x$.  Thus $\overline x$ means the element $x$ expressed with the variable $t$ and $y$ instead of with $\theta$ and $\eta$.  We remark that the maps $\iota$ and $\chi$ extend to finite algebraic extensions of $\bK$ and $K$ respectively.

Let $\ord_\infty$ be the valuation of $K$ at the infinite place, and let $\deg := -\ord_\infty$, both normalized so that
\[\deg(\theta) = 2, \quad \deg(\eta) = 3.\]
Define an absolute value on $K$ by setting $|g| = q^{\deg(g)}$ for $g\in K$.  Also define $\ord_\infty$, $\deg$ and $|\cdot|$ on $\bK$ similarly.  Let $K_\infty$ be the completion of $K$ at the infinite place, and let $\C_\infty$ be the completion of an algebraic closure of $K_\infty$.  Designate the point $\Xi= (\theta,\eta)\in E(K)$.

Extend the absolute value on $\C_\infty$ to a seminorm on $M = \langle m_{i,j} \rangle \in \Mat_{\ell\times m}(\C_\infty)$ as in \cite[\S 2.2]{PLogAlg}  by defining
\[|M| = \max_{i,j}(|m_{i,j}|).\]
Note for $c\in \C_\infty$ and $M,N \in \Mat_{\ell\times m}(\C_\infty)$ that
\[|cM| = |c|\cdot |M|,\quad |M+N|\leq |M|+|N|,\]
and for matrices $M\in \Mat_{k\times \ell}(\C_\infty)$ and $N\in \Mat_{\ell\times m}(\C_\infty)$ that
\[|MN|\leq |M|\cdot |N|,\]
but that the seminorm is not multiplicative in general.

In order to define a sign function, we first note that as an $\F_q$-vector space, $\bA$ has a basis $\bA = \Span_{\F_q}(t^i,t^jy),$ for $i,j\geq 0$ where each term has a unique degree.  Thus, when expressed in this basis, an element $a\in \bA$ has a leading term which allows us to define
\[\sgn:\bA\setminus \{0\} \to \F_q^\times,\]
by letting $\sgn(a)\in \F_q^\times$ be the coefficient of the leading term of $a\in \bA\setminus \{0\}$.  This sign function extends naturally to $\bK$.  Define a sign function analogously on $A$ and $K$, which we also call $\sgn$.  Then, for any field extension $L/\F_q$, the coordinate ring of $E$ over $L$ is $L[t,y] = L\otimes_{\F_q} \bA$, and using the same notion of leading term, we define a group homomorphism
\[\tsgn:L(t,y)^\times \to L^\times,\]
which extends the functin $\sgn$ on $\bK$.

Now, let $L/\F_q$ be an algebraically closed extension of fields.  Define $\tau:L\to L$ to be the $q$th power Frobenius map and define $L[\tau]$ as the ring of twisted polynomials in $\tau$, subject to the relation for $c\in L$
\[\tau c = c^q\tau.\]

Define the Frobenius twisting automorphism of $L[t,y]$ for $g = \sum c_{j,k} t^jy^k$ as
\begin{equation}\label{twistdef}
g\twist = \sum c_{j,k}^q t^jy^k,
\end{equation}
where $g\twisti$ means the $i$th iteration of twisting.  The twisting operation also extends naturally to matrices in $\Mat_{\ell\times m}(L(t,y))$ by twisting entry-wise.  We use this notion of twisting to define the ring $\Mat_n(L)[\tau]$ as the non-commutative ring of polynomials in $\tau$ subject to the relation $\tau M = M\twist \tau$ for $M\in \Mat_n(L)$.  In the setting of Anderson $\bA$-modules, we view $\Mat_n(L)[\tau]$ as a ring of operators acting on $L^n$ for $n\geq 1$ via twisting, i.e. for $\Delta = \sum M_i \tau^i$, with $M_i\in \Mat_n(L)$ and $\ba \in L^n$,
\begin{equation}\label{Deltaaction}
\Delta(\ba) = \sum M_i \ba\twisti.
\end{equation}
Further, for $X\in E(L)$, we define $X\twist = \Fr(X)$, where $\Fr:E\to E$ is the $q$th power Frobenius isogeny.  We extend twisting to divisors in the obvious way, noting that for $g\in L(t,y)$
\[\divisor(g\twist) = \divisor(g)\twist.\]

We define the Tate algebra for $c\in A$,
\begin{equation} \label{Tatealgs}
   \TT_c = \biggl\{ \sum_{i=0}^\infty b_i t^i \in \power{\C_\infty}{t} \biggm| \big\lvert c^i b_i \big\rvert \to 0 \biggr\},
\end{equation}
where $\TT_c$ is the set of power series which converge on the closed disk of radius $|c|$.  For convenience, we set $\TT:= \TT_1$.  Define the Gauss norm $\lVert \cdot \rVert_c$ for vectors of functions $\bh = \sum \bd_i t^i \in \TT_c^n$ for some fixed dimension $n>0$ with $\bd_i\in \C_\infty^n$ by setting
\[\lVert \bh \rVert_c = \max_i |c^i\bd_i|,\]
where $|\cdot|$ is the seminorm described above.  Extend this norm to $\TT_c[y]^n$ for $\bh_1,\bh_2\in \TT_c^n$ by setting $\lVert \bh_1+y\bh_2 \rVert_c = \max(\lVert \bh_1 \rVert_c ,\lVert \eta \bh_2 \rVert_c )$.  Note that each of these algebras are complete under their respective norms.  Using the definition given from \cite[Chs.~3--4]{FresnelvdPut}, we note that the two rings $\TT[y]$ and $\TT_\theta[y]$ are affinoid algebras corresponding to rigid analytic affinoid subspaces of $E/\C_\infty$.  Let $\cE$ be the rigid analytic variety associated to $E$ and let $\cU\subset \cE$ be the inverse image under $t$ of the closed disk of radius $|\theta|$ in $\C_\infty$ centered at 0.  Then observe that $\cU$ is the affinoid subvariety of $\cE$ associated to $\TT_\theta[y]$, and that Frobenius twisting extends to $\TT$ and $\TT[y]$ and their fraction fields.  As proved in \cite[Lem.~3.3.2]{P08}, $\TT$ has $\F_q[t]$ and $\TT[y]$ has $\bA$ as its fixed ring under twisting.

We extend the action of $\Mat_n(L)[\tau]$ on $L^n$ described in \eqref{Deltaaction} to an action of $\Mat_n(\TT[y])[\tau]$ on $\TT[y]^n$ in the natural way.

\section{Tensor Powers of A-motives}\label{S:Amotives}
We briefly review the theory of $\bA$-motives and dual $\bA$-motives corresponding to rank 1 sign-normalized Drinfeld-Hayes modules as set out in \cite[\S 3]{GP16}.  First note that we can pick a unique point $V$ in $E(H)$ whose coordinates have positive degree such that $V$ satisfies the equation on $E$
\begin{equation}\label{Vequation}
(1-\Fr)(V) = V-V\twist = \Xi,
\end{equation}
If we set $V = (\alpha,\beta)$, then $\deg(\alpha) = 2$ and $\deg(\beta) = 3$ and $\sgn(\alpha) = \sgn(\beta) = 1$.  Define $H$ to be the Hilbert class field of $K$, which equals $H = K(\alpha,\beta)$.  There is a unique function in $H(t,y)$, called the shtuka function, with $\tsgn(f) = 1$ and with divisor
\begin{equation}\label{fdivisor}
\divisor(f) = (V\twist) - (V) + (\Xi) - (\infty).
\end{equation}
We can write
\begin{equation} \label{fdef}
  f = \frac{\nu(t,y)}{\delta(t)} = \frac{y - \eta - m(t-\theta)}{t-\alpha} = \frac{y + \beta + c_1\alpha + c_3 - m(t-\alpha)}{t-\alpha},
\end{equation}
where $m$ is the slope between the collinear points $V\twist, -V$ and $\Xi$, and $\deg(m) = q$.  We see
\begin{gather} \label{nudiv}
\divisor(\nu) = (V\twist) + (-V) + (\Xi) - 3(\infty), \\
\label{deltadiv}
\divisor(\delta) = (V) + (-V) - 2(\infty).
\end{gather}
Let $L/K$ be an algebraically closed field, and let $U = \Spec L[t,y]$ be the affine curve $(L \times_{\F_q} E) \setminus \{ \infty \}$.

We let
\[
  M_0 = \Gamma(U, \cO_{E}(V))  = \bigcup_{i \geq 0} \cL((V) + i(\infty)),
\]
where $\cL((V) + i(\infty))$ is the $L$-vector space of functions $g$ on $E$ with $\divisor(g) \geq -(V) - i(\infty)$.  We make $M_0$ into a left $L[t,y,\tau]$-module by letting $\tau$ act by
\[
  \tau g = f g^{(1)}, \quad g \in M_0,
\]
and letting $L[t,y]$ act by left multiplication.  We find that $M_0$ is a projective $L[t,y]$-module of rank~$1$ as well as a free $L[\tau]$-module of rank~$1$ with basis~$\{ 1 \}$.  Define the dual $\bA$-motive
\begin{equation} \label{dualt}
  N_0 = \Gamma\bigl(U, \cO_E(-(V^{(1)}))\bigr) \subseteq L[t,y].
\end{equation}
If we let $\sigma = \tau^{-1}$, then we can define a left $L[t,y,\sigma]$-module structure on $N_0$ by setting
\[
  \sigma h = f h^{(-1)}.
\]
With this action $N_0$ is a dual $\bA$-motive in the sense of Anderson~\cite[\S 4]{HJ16}, and we note that $N_0$ is an ideal of $L[t,y]$ and that it is a free left $L[\sigma]$-module of rank~$1$ generated by $\delta^{(1)}$ (see \cite[\S 3]{GP16} for proofs of these facts).

A Drinfeld $\bA$-module over $L$ is an $\F_q$-algebra homomorphism
\[
  \rho : \bA \to L[\tau],
\]
such that for all $a \in \bA$,
\[
  \rho_a = \iota(a) + b_1 \tau + \dots + b_n \tau^n.
\]
The rank $r$ of $\rho$ is the unique integer such that $n = r \deg a$ for all $a$.  Thus, rank 1 sign-normalized means that we require $r=1$ and that $b_n = \sgn(a)$.

For the Drinfeld $\bA$-module $\rho$, denote the exponential function associated to $\rho$ as $\exp_\rho(z) = \sum_{i=0}^\infty \frac{z^{q^i}}{d_i}$ with $d_0=1$, and denote it's period lattice as $\Lambda_\rho$.  Theorem 4.6 from \cite{GP16} states that $\Lambda_\rho$ is a rank 1 free $A$-module and is generated by the fundamental period
\begin{equation}\label{1dimperiod}
  \pi_{\rho} = - \frac{\xi^{q/(q-1)}}{\theta^q - \alpha} \prod_{i=1}^{\infty}
  \left(\frac{1 - \dfrac{\theta}{\alpha^{q^i} \mathstrut }}{1 - \biggl( \dfrac{m}{m\theta - \eta}\biggr)^{q^{i \mathstrut}} \cdot \theta + \biggl( \dfrac{1}{m\theta - \eta} \biggr)^{q^{i\mathstrut}}\cdot \eta}\right),
\end{equation}
where $\xi = -(m + \beta/\alpha)$.

We now proceed to developing the theory for $n$-dimensional tensor powers of $\bA$-motives and dual $\bA$-motives.  This generalizes the theory for the $n$-dimensional $t$-motives for the Carlitz module (see \cite[\S 3.6]{PLogAlg}).  For a fixed dimension $n\geq 1$, we define the $n$-fold tensor power of $M_0$,
\[M_0\on = M_0\otimes_{L[t,y]} \dots \otimes_{L[t,y]}M_0,\]
and similarly for $N_0\on$.  We wish to analyze $M_0\on$ and $N_0\on$ and identify them as a spaces of functions over $U$.

\begin{proposition}\label{P:TensorProd}
For $n\geq 1$, we have the following $L[t,y]$-module isomporphisms
\[M_0\on \isom \Gamma (U,\cO_E(nV)) \quad \text{and}\quad N_0\on \isom \Gamma (U,\cO_E(-nV\twist)).\]
\end{proposition}
\begin{proof}
Define the map
\[\psi: M_0\otimes_{L[t,y]} \dots \otimes_{L[t,y]}M_0 \to \Gamma (U,\cO_E(nV))\]
on simple tensors for $a_i\in M_0$ as
\[a_1\otimes \dots \otimes a_n \mapsto a_1\cdots a_n.\]
Looking at divisors, one quickly sees that $a_1\cdots a_n$ is indeed in $\Gamma (U,\cO_E(nV))$ as desired.  Then it follows quickly from Proposition 5.2 of \cite{Hartshorne} that the map $\psi$ is an $L[t,y]$-module isomorphism.  The proof that $N_0\on \isom \Gamma (U,\cO_E(-nV\twist))$ follows similarly.
\end{proof}
From here on forward, we will denote
\begin{equation}\label{MNdef}
M := M_0\on = \Gamma (U,\cO_E(nV)),\quad N :=N_0\on= \Gamma (U,\cO_E(-nV\twist)).
\end{equation}
We turn $M$ into an $L[t,y,\tau]$-module and $N$ into an $L[t,y,\sigma]$-module by defining the action for $a\in M$ and $b \in N$ as
\begin{equation}\label{tauaction}
\tau a = f^na\twist\quad \text{and}\quad \sigma b = f^n b\twistinv.
\end{equation}

\begin{remark}
The $\tau$ action defined on $M$ in \eqref{tauaction} is the same as the diagonal action on $M_0\on$, namely for $a_i \in M_0$
\[\psi(\tau(a_1\otimes \dots \otimes a_n)) =\psi (\tau a_1\otimes \dots \otimes \tau a_n) = \psi(f a_1 \otimes \dots \otimes f a_n) = f^n \psi(a_1 \otimes \dots \otimes a_n).\]
Thus the map $\psi$ from Proposition \ref{P:TensorProd} is actually an $L[t,y,\tau]$-module isomorphism.
\end{remark}

For a fixed dimension $n \geq 2$, we define a set of functions which generate $M$ as a free $L[\tau]$-module and define a second set of functions which generate $N$ as a free $L[\sigma]$-module.  We remark that for the case of $n=1$, the present considerations do reduce to those detailed in \cite[\S 3]{GP16} for motives attached to rank 1 Drinfeld modules, but for ease of exposition we assume that $n\geq 2$.  Define a sequence of functions $g_i \in M$ for $1\leq i\leq n$ with $\tsgn(g_i)=1$ and with divisors
\begin{equation}\label{gidivisor}
\divisor(g_{j}) = -n(V) +  (n-j)(\infty) + (j-1)(\Xi) + ([j-1]V\twist + [n-(j-1)]V),
\end{equation}
and define functions $h_i \in N$ with $\tsgn(h_i)=1$ and with divisors
\begin{equation}\label{hidivisor}
\divisor(h_j) = n(V\twist) - (n+j)(\infty) + (j-1)(\Xi) + (-[n - (j - 1)]V\twist-[j-1]V).
\end{equation}
Note that the divisors in \eqref{gidivisor} and \eqref{hidivisor} are principal by \eqref{Vequation}, that the functions $g_i$ and $h_i$ are uniquely defined because of the $\tsgn$ condition and that $g_i,h_i \in H(t,y)$.
\begin{proposition}\label{P:TauBasis}
For $n\geq 2$, the set of functions $\{g_i\}_{i=1}^n$ generate $M$ as a free $L[\tau]$-module and the set of functions $\{h_i\}_{i=1}^n$ generate $N$ as a free $L[\sigma]$-module.
\end{proposition}
\begin{proof}
First observe that by the definition of the action of $\tau$ from \eqref{tauaction} that the $L$-vector space generated by the functions $\tau ^j g_i$ for $1\leq i\leq n$ and $j\geq 0$ is contained in $M$.  Then observe that each of the functions $g_i$ lives in the 1-dimensional Riemann-Roch space
\[g_i \in \mathcal{L}(n(V) -  (n-i)(\infty) - (i-1)(\Xi)).\]
Further, by the Riemann-Roch theorem
\[\mathcal{L}(n(V)) = \bigcup_{j=1}^n \mathcal{L}(n(V) -  (n-j)(\infty) - (j-1)(\Xi)),\]
so that $\mathcal{L}(n(V))$ is equal to the $L$-span of the functions $g_i$.  Finally, observe that
\[\deg(\tau^j g_i) =\deg((ff\twist \dots f\twistk{j-1})^n g_i\twistk{j})= (j-1)n + i,\]
so that the degree of each $\tau^j g_i$ is unique and that these degrees includes each nonnegative integer, thus
\[M = \bigcup_{i=1}^\infty \mathcal{L}(n(V) +  i(\infty))\]
is equal to the $L$ span of the set $\{\tau^j g_i\}$ for $1\leq i\leq n$ and $j\geq 0$.  The proof for the $\sigma$-basis of the dual $\bA$-motive $N$ follows similarly, once we note that each $h_i$ belongs to a 1-dimensional Riemann-Roch space
\[h_i \in \mathcal{L}(-n(V\twist) + (n+j)(\infty) - (j-1)(\Xi)).\]
We leave the details of this case to the reader.
\end{proof}
When it is convenient, we will extend the definitions of the functions $g_i$ and $h_i$ for $i>n$ by writing $i=jn+k$, where $1\leq k\leq n$, and then denoting,
\begin{equation}\label{gihigher}
g_i := \tau^j(g_k) = (ff\twist \dots f\twistk{j-1})^n g_k\twistk{j}\quad \text{and}\quad h_i := \sigma^j(h_k) = (ff\twistinv \dots f\twistk{1-j})^n h_k\twistk{-j}.
\end{equation}

\begin{remark}
The $\bA$-motive $N$ is dual to the $\bA$-motive $M$ in a precise sense as outlined in \cite[Prop. 4.3]{HJ16}.  But, as we do not need this for the rest of the paper, we omit the details.  We do, however, record a lemma about the relationship between the functions $g_i$ and $h_i$ which we will need later.
\end{remark}
\begin{lemma}\label{L:basisduality}
We obtain the following identities of functions for $1\leq j\leq n-1$
\[g_1h_1\twistinv = t-t([n]V),\]
\[g_{j+1}h_{n-(j-1)} = f^n\cdot (t-t([j]V\twist+[n-j]V)).\]
\end{lemma}
\begin{proof}
The first identity is proved trivially, simply by comparing divisors from \eqref{gidivisor} and \eqref{hidivisor}, and noting that
\[\tsgn(g_{1}) = \tsgn(h_{1}) = \tsgn(t) = 1.\]
The second follows similarly, noting that for $1\leq j\leq n-1$
\[\divisor(g_{j+1}h_{n-(j-1)}) = \divisor(f^n\cdot (t-t([j]V\twist+[n-j]V))),\]
and thus the two sides are equal up to a multiplicative constant.  Then, since
\[\tsgn(g_{j+1}h_{n-(j-1)}) = \tsgn(f^n\cdot (t-t([j]V\twist+[n-j]V))) = 1,\]
the equality of functions follows.
\end{proof}

\begin{remark}
The $L[t,y,\tau]$-module $N$ and the $L[t,y,\sigma]$-module $M$ with the actions described in \eqref{tauaction} is an $\bA$-motive and a dual $\bA$-motive, respectively, in the precise sense described by Anderson (see \cite[\S 4]{HJ16}).  Because we do not require this fact going forward in the present paper, we omit the details.
\end{remark}

\section{Anderson A-modules}\label{S:Amodules}
In this section we show how to construct an Anderson $\bA$-module from the $\bA$-motive $M$ of the previous section.  An $n$-dimensional Anderson $\bA$-module is a map $\rho:\bA\to \Mat_n(L)[\tau]$, such that for $a\in \bA$
\[\rho_a = d[a] + A_1 \tau + \dots + A_m\tau^m,\]
where $d[a] = \iota(a)I + N$ for some nilpotent matrix $N\in \Mat_n(L)$, and we remark that $d:\bA\to \Mat_n(L)$ is a ring homomorphism.  The map $\rho\on$ describes an action of $\bA$ on the underlying space $L^n$ in the sense defined in \eqref{Deltaaction}, allowing us to view $L^n$ as an $\bA$-module.  Anderson $\bA$-modules are a generalization of the $t$-modules introduced by Anderson in \cite{And86}; they are studied thoroughly in \cite[\S 5]{HJ16}.

\begin{example}[Tensor Powers of the Carlitz Module]\label{Carlitztensorpowers}
For $\bA = \F_q[t]$, define an $n$-dimensional Anderson $\bA$-module $C\on:\F_q[t]\to \Mat_n(\F_q[\theta])[\tau]$ (called in this situation a $t$-module, with the normalization $\deg(t)=1$) by setting
\[C\on_t = (\theta I + N_1) + E_1\tau.\]
Thus, for $\bz\in L^n$,
\[C\on_t(\bz) = (\theta I + N_1)\bz + E_1 \bz\twist,\]
and we extend $C\on$ to all of $\bA$ by setting $C_{t^m} = C_t^m$ and using $\F_q$-linearity.  The map $C\on$ is an Anderson $\bA$-module and is called the $n$th tensor power of the Carlitz module.
\end{example}
Work by Anderson in \cite[Thm. 3]{And86} for the $\bA = \F_q[t]$ case, then later by B\"ockle and Hartle in \cite[\S 8.6]{BH07} for the more general rings $\bA$, shows that associated to every Anderson $\bA$-module, there is a unique, $\F_q$-linear power series, which we label
\[
\Exp _\rho(\bz) = \sum_{i=0}^\infty Q_i \bz\twisti \in \power{\Mat_n(\C_\infty)}{\bz},
\]
defined so that $Q_0 = I$ and that for all $a\in \bA$ and $\bz\in \C_\infty^n$
\begin{equation}\label{expfunctionalequation}
\Exp _\rho(d[a]\bz) = \rho_a(\Exp _\rho(\bz)).
\end{equation}
We call $\Exp _\rho$ the exponential function associated to $\rho $, and note that it is entire on $\C_\infty^n$.  We also define the logarithm function associated to $\bA$ to be the formal inverse of $\Exp _\rho$.  We label its coefficients
\[
\Log _\rho(\bz) = \sum_{i=0}^\infty P_i \bz\twisti \in \power{\Mat_n(\C_\infty)}{\bz},
\]
and note that $\Log _\rho$ also satisfies a functional equation for each $a\in \bA$
\begin{equation}\label{logfunctionalequation}
\Log _\rho(\rho _a(\bz)) = d[a]\Log _\rho(\bz).
\end{equation}
The function $\Log _\rho$ has a finite radius of convergence in $\C_\infty^n$, which we denote $r_L$.

Given the $\bA$-motive $M$ and the dual $\bA$-motive $N$ defined in \eqref{P:TensorProd}, we now describe how to use these motives to define an Anderson $\bA$-module.    This method generalizes a technique of Thakur \cite[0.3.5]{Thakur93} for Drinfeld modules, and also has roots in unpublished work of Anderson (see \cite[\S 5.2]{HJ16}).  We also refer the reader to \cite{BP16} for a thorough account of the functoriality of this process in the case of $t$-modules.  We begin by defining the $t$- and $y$-action of the $\bA$-module, from which the rest of the action of $A$ can be defined.  These actions are defined in terms of constants coming from the functions $g_i$ and $h_i$ from \eqref{gidivisor}.

\begin{proposition}\label{P:definingequations}
There exist constants $a_i,b_i,y_i,z_i\in H$ such that we can write for $1\leq i\leq n$
\begin{align*}
tg_i &= \theta g_i + a_ig_{i+1} + g_{i+2},\\
y g_i & = \eta g_i + y_i g_{i+1} + z_i g_{i+2} + g_{i+3},\\
th_i &= \theta h_i + b_ih_{i+1} + h_{i+2},
\end{align*}
where we recall the definitions of $g_i$ and $h_i$ for $i>n$ from \eqref{gihigher}.
\begin{proof}
Note that $tg_i\in M$, and hence we can write
\[tg_i = c_1g_1 + c_2g_2 + \dots + c_mg_m,\]
for $c_i\in \C_\infty$.  Examining the order of vanishing at $\infty$ of $g_j$ from \eqref{gidivisor} and recalling that $t$ has a pole of order $2$ at $\infty$, we see that $c_j=0$ for $j<i$ and $j>i+2$.  So
\[tg_i = c_ig_i + c_{i+1}g_{i+1} + c_{i+2}g_{i+2}.\]
Then, noting that $\tsgn(g_i) = \tsgn(t) = 1$ and evaluating both sides at $\Xi$ shows  that $c_{i+2} = 1$ and that $c_i = \theta$, respectively.  Further, all the functions $g_i$ are in $H(t,y)$, and so the constants $c_i$ are as well, which finishes the proof of the first equation.  The proofs of the other two equations follow similarly; we leave the details to the reader.
\end{proof}
\end{proposition}
Given the relationship between the basis elements $g_i$ and $h_j$ described in Lemma \ref{L:basisduality}, we also expect the coefficients $a_i$ and $b_j$ to be related.
\begin{proposition}\label{P:coefficientidentities}
For the coefficients defined in Proposition \ref{P:definingequations}, for $j\leq n-1$,
\[a_j = b_{n-j}\quad \text{and} \quad a_n = b_n^q.\]
\end{proposition}
\begin{proof}
From Proposition \ref{P:definingequations} we calculate that
\begin{equation}\label{mixedequality}
0 = (\theta - t)\left (\frac{g_j}{g_{j+2}} - \frac{h_{n-j}}{h_{n-j+2}}\right ) + a_j\frac{g_{j+1}}{g_{j+2}} - b_{n-j} \frac{h_{n-j+1}}{h_{n-j+2}}.
\end{equation}
Then using Lemma \ref{L:basisduality} yields the equality of functions
\[\frac{h_{n-j}}{h_{n-j+2}} = \frac{t - t\left ([j+1]V\twist + [n - (j+1)]V\right )}{t - t\left ([j-1]V\twist + [n - (j-1)]V\right )}\cdot \frac{g_j}{g_{j+2}},\]
and so \eqref{mixedequality} becomes
\[(\theta - t)\left (\frac{g_j}{g_{j+2}}\right )\left (1 - \frac{t - t\left ([j+1]V\twist + [n - (j+1)]V\right )}{t - t\left ([j-1]V\twist + [n - (j-1)]V\right )}\right ) = -a_j\frac{g_{j+1}}{g_{j+2}} + b_{n-j} \frac{h_{n-j+1}}{h_{n-j+2}}.\]
From \eqref{gidivisor} and \eqref{hidivisor} we quickly see that
\[\deg\left ((\theta - t)\left (\frac{g_j}{g_{j+2}}\right )\left (1 - \frac{t - t\left ([j+1]V\twist + [n - (j+1)]V\right )}{t - t\left ([j-1]V\twist + [n - (j-1)]V\right )}\right )\right ) = 0,\]
whereas
\[\deg\left (a_j\frac{g_{j+1}}{g_{j+2}}\right ) = \deg\left (b_{n-j} \frac{h_{n-j+1}}{h_{n-j+2}}\right ) = -1.\]
Then, since $\tsgn(g_i) = \tsgn(h_i) = 1$, in order for the degree on the left hand side to match the degree on the right hand side, we must have that $a_j = b_{n-j}$ for $j\leq n-1$.  Similar analysis of the equations from Proposition \ref{P:definingequations} shows that $a_n = b_n^q$; we leave the details to the reader.
\end{proof}

We begin defining the Anderson $\bA$-module associated to $M$, which we call the $n$th tensor power of the Drinfeld module $\rho$ associated to $M_0$, by defining
\begin{equation}\label{taction}
\rho\on_t := d[\theta] + E_\theta \tau :=
\left (\begin{matrix}
\theta & a_1 & 1  & 0 & \hdots & 0 & 0 & 0 \\
0 & \theta & a_2  & 1 & \hdots & 0 & 0 & 0 \\
0 & 0 & \theta  & a_3 & \hdots & 0 & 0 & 0 \\
\vdots & \vdots &  \vdots &\vdots & \ddots &  \vdots & \vdots & \vdots \\
0 & 0 & 0  & 0 & \hdots & \theta & a_{n-2} & 1\\
0 & 0 & 0  & 0 & \hdots & 0 & \theta & a_{n-1}\\
0 & 0 & 0  & 0 & \hdots & 0 & 0 & \theta
\end{matrix}\right ) + 
\left (\begin{matrix}
0 & 0 & 0 & \hdots & 0 \\
\vdots & \vdots &  \vdots & \ddots & \vdots \\
0 & 0 & 0 & \hdots & 0\\
1 & 0 & 0 & \hdots & 0\\
a_n & 1 & 0 & \hdots & 0
\end{matrix}\right )\tau
\end{equation}
and
\begin{equation}\label{yaction}
\rho\on_y := d[\eta] + E_\eta \tau :=
\left (\begin{matrix}
\eta & y_1 & z_1 & 1  & 0 & \hdots & 0 & 0 & 0 \\
0 & \eta & y_2 & z_2 & 1 & \hdots & 0 & 0 & 0 \\
0 & 0 & \eta  & y_3 & z_3 & \hdots & 0 & 0 & 0 \\
\vdots & \vdots &  \vdots  & \vdots & \vdots& \ddots & \vdots & \vdots& \vdots \\
0 & 0 & 0  & 0 &0 & \hdots & \eta & y_{n-2} & z_{n-2}\\
0 & 0 & 0  & 0 & 0 & \hdots & 0 & \eta & y_{n-1}\\
0 & 0 & 0  & 0 &  0 & \hdots & 0 & 0 & \eta
\end{matrix}\right ) + 
\left (\begin{matrix}
0 & 0 & 0  & 0  & \hdots & 0 \\
\vdots & \vdots &  \vdots & \vdots & \ddots & \vdots \\
0 & 0 & 0 & 0 &  \hdots & 0\\
1 & 0 & 0 & 0 &  \hdots & 0\\
z_{n-1} & 1 & 0 & 0 &  \hdots & 0\\
y_n & z_n & 1 & 0 & \hdots & 0
\end{matrix}\right )\tau,
\end{equation}
where $a_i$, $y_i$ and $z_i$ are given in Proposition \ref{P:definingequations}.

To simplify notation later, we define strictly upper triangular matrices $N_\theta$ and $N_\eta$ by
\begin{equation}\label{Nthetadef}
N_\theta = d[\theta] - \theta I\quad\text{and}\quad N_\eta = d[\eta] -   \eta I .
\end{equation}
With the definitions of $\rho\on_t$ and $\rho\on_y$, we define the $\F_q$-linear map
\[\rho\on_a: \bA \to \Mat_n(H[\tau])\]
for any $a\in \bA$ by writing $a=\sum c_it^i + y\sum d_it^i$ with $c_i,d_i\in \F_q$, and extending using linearity and the composition of maps $\rho\on_{t^a} = (\rho\on_t)^a$.  A priori, the map $\rho$ is just an $\F_q$-linear map, but we will shortly show that it actually is an $\F_q$-algebra homormophism and defines an Anderson $\bA$-module.

\begin{remark}
In general the coefficients $a_i$, $y_i$ and $z_i$ are not integral over $H$, which could lead to our chosen model for $\rho\on$ having bad reduction over certain places of $A$.  We suspect that it is possible to choose a normalization which has everywhere good reduction, but this would come at the expense of having more complicated formulas e.g. not having $1$'s across the last non-zero super diagonals of $\rho\on_t$ and $\rho\on_y$.
\end{remark}

Our main strategy for showing that the map $\rho\on$ is actually an Anderson $\bA$-module involves constructing a second Anderson $\bA$-module $\rho'$ using techniques of Hatle and Juschka, then showing that the maps $\rho\on$ and $\rho'$ align.  In what follows, for convenience, we fix the algebraically closed field $L$ from \S\ref{S:Amotives} to be $\C_\infty$, although we remark that much of the theory applies equally to any algebraically closed field.  For $g\in N = \Gamma(U,\mathcal O_E(-nV\twist))$, define the map
\[\varepsilon: N \to \C_\infty^n,\]
by writing $g$ in the basis for the dual $\bA$-motive arranged as
\begin{align}\label{gcoefficients1}
\begin{split}
g & = d_{1,0}h_1 + d_{1,1}h_1\twistinv f^n + \dots + d_{1,m}h_1\twistk{-m}(f f^{(-1)} \cdots f^{(-m+1)})^n\\
&+ d_{2,0}h_2 + d_{2,1}h_2\twistinv f^n + \dots + d_{2,m}h_2\twistk{-m}(f f^{(-1)} \cdots f^{(-m+1)})^n\\
&\quad\vdots\\
&+ d_{n,0}h_n + d_{n,1}h_n\twistinv f^n + \dots + d_{n,m}h_n\twistk{-m}(f f^{(-1)} \cdots f^{(-m+1)})^n,
\end{split}
\end{align}
where $d_{i,j}\in \C_\infty$ and at least one of the $d_{i,m}$ is non-zero, then defining
\begin{equation}\label{varepsdefinition}
\varepsilon(g) = 
\left (\begin{matrix}
d_{n,0}\\
d_{n-1,0}\\
\vdots\\
d_{1,0}
\end{matrix}\right ) +
\left (\begin{matrix}
d_{n,1}\\
d_{n-1,1}\\
\vdots\\
d_{1,1}
\end{matrix}\right )\twist+\dots + 
\left (\begin{matrix}
d_{n,m}\\
d_{n-1,m}\\
\vdots\\
d_{1,m}
\end{matrix}\right )\twistk{m}.
\end{equation}
One observes immediately from the definition that $\varepsilon$ is $\F_q$-linear.  We then obtain a proposition similar to Lemma 3.6 from \cite{GP16}.

\begin{proposition}\label{P:varepsilon}
The map $\varepsilon : N \to \C_\infty^n$ is surjective and
\[
\ker(\varepsilon) = (1 - \sigma)N  =  \bigl\{ g \in N \mid g = h^{(1)} - f^nh\ \textup{for some\ } h
\in \Gamma(U,\cO_{E}(-n(V))) \bigr\}.
\]
\end{proposition}
\begin{proof}
For $h \in \Gamma(U,\cO_E(-n(V)))$, we have $h^{(1)} \in N$ and $\sigma(h^{(1)}) = f^nh$, so the two objects on the right are the same.  Also, if we write $h\twist$ using the basis and notation from \eqref{gcoefficients1}, after a short calculation we find that $\varepsilon(h^{(1)}) = \varepsilon(f^nh)$, and thus $(1-\sigma)N \subseteq \ker(\varepsilon)$.  To show that $\ker(\varepsilon) \subseteq (1-\sigma)N$, we note that by the proof of Proposition \ref{P:TauBasis} each function on the right hand side of \eqref{gcoefficients1} has unique degree.  Then for $g \in \ker(\varepsilon)$, we can construct a function $h\in \Gamma(U,\cO_{E}(-n(V)))$ satisfying $g = h^{(1)} - f^nh$ through the following process.  We first note that degree considerations force $\deg(h) = \deg(g)  - n$, then we observe that $h\twist \in N$ and so we can write $h\twist$ in terms of the same basis used in \eqref{gcoefficients1} with coefficients $d_{i,j}'\in \C_\infty$.  Next, we set $g = h^{(1)} - f^nh$ and compare coefficients of equal degree terms on each side.  The fact that $g\in \ker(\varepsilon)$ allows us to solve for the coefficients $d_{i,j}'$ uniquely in terms of the coefficients of $g$, which proves that such a function $h\in \Gamma(U,\cO_{E}(-n(V)))$ exists.  We leave the details of the calculation to the reader.
\end{proof}
We then combine Proposition \ref{P:varepsilon} with a theorem of Hartl and Juschka \cite[Proposition 5.6]{HJ16} to obtain the following proposition.
\begin{proposition}\label{P:AndersonbA-module}
The map $\rho\on$ is an Anderson $\bA$-module.
\end{proposition}
\begin{proof}
Since $N$ is free of rank $n$ and finitely generated as a $\C_\infty[\sigma]$-module, the quotient module $N/(1-\sigma)N$ is isomorphic as an $\F_q$-vector space to $\C_\infty^n$.  We choose a basis for $N/(1-\sigma)N$ consisting of the functions $\overline{h_i}$, the images of $h_i$ under the quotient map, then observe by Proposition \ref{P:varepsilon}, that this isomorphism is given by $\varepsilon$.  This gives rise to the following commutative diagram, 
\begin{equation}\label{Amotivediagram}
\begin{diagram}
N/(1-\sigma)N &\rTo^{\varepsilon} & \C_\infty^n  \\
\dTo^{a} & & \dTo_{\rho_a'} \\
N/(1-\sigma)N &\rTo^{\varepsilon} & \C_\infty^n
\end{diagram}
\end{equation}
where the vertical map on the left is multiplication by $a\in \bA$ and the vertical map on the right is the map induced by multiplication by $a$ under the isomorphism $\varepsilon$.  This diagram describes an action of $\bA$ on the space $\C_\infty^n$, and a priori, the induced action $\rho_a'$ is in $\End_{\F_q}(\C_\infty^n)$.  However, proposition 5.6 of Hartl and Juschka \cite{HJ16} shows that $\rho_a'$ is actually in $\Mat_n(\C_\infty[\sigma])$ and that it defines an Anderson $\bA$-module.  To write down the action of $\rho_a'$, we only need to analyze the action of $a$ on the basis elements $h_i$ (we drop the overline notation, since there is no confusion), and since $\bA$ is generated as an algebra by $t$ and $y$, we only need to consider the action of $t$ and $y$ on the basis elements.  We first note that for $1\leq i\leq n-2$ and $d_i\in \C_\infty$ by Proposition \ref{P:definingequations} and by the definition of $\varepsilon$ in \eqref{varepsdefinition}
\[\varepsilon(td_ih_i) = \varepsilon(d_i(\theta h_i + b_ih_{i+1} + h_{i+2})) = d_i(0,\dots,0,1,b_i,\theta,0,\dots,0)^\top\]
while we also have
\[
\varepsilon(td_{n-1}h_{n-1})= \varepsilon(d_{n-1}(\theta h_{n-1} + b_{n-1}h_{n} + \sigma(h_{1}))) = d_{n-1}(b_i,\theta,0,\dots,0)^\top + d_{n-1}^q(0,\dots,0,1)^\top
\]
\[
\varepsilon(td_{n}h_{n})= \varepsilon(d_{n}(\theta h_{n} + b_{n}\sigma(h_{1}) + \sigma(h_{2}))) = d_{n-1}(\theta,0,\dots,0)^\top + d_{n-1}^q(0,\dots,0,1,b_n^q)^\top.
\]
Using the identities from Proposition \ref{P:coefficientidentities}, and piecing this all together, yields
\[\varepsilon(t(d_1h_1 + \dots + d_nh_n)) = (d[\theta] + E_\theta \tau)(d_1,\dots, d_n)^\top = \rho\on_t(d_1,\dots, d_n)^\top.\]
Similar analysis gives
\[\varepsilon(y(d_1h_1 + \dots + d_nh_n)) = \rho\on_y(d_1,\dots, d_n)^\top.\]
Therefore, the operators $\rho_t' = \rho\on_t$ and $\rho_y' = \rho\on_y$, and we see that the map $\rho$ defined in \eqref{taction} is actually an $\bA$-module homomorphism and defines an Anderson $\bA$-module.
\end{proof}

\begin{remark}
We comment that it is likely possible to prove that $\rho$ is an Anderson $\bA$-module by appealing to Mumford's work in \cite{Mumford78} as does Thakur in \cite{Thakur93}, however, we prefer the approach inspired by Hartl and Juschka in \cite{HJ16}.
\end{remark}

Having proved that $\rho\on$ is an Anderson $\bA$-module, we will label the exponential and logarithm function associated to $\rho\on$ as
\begin{equation}\label{Exptensor}
\Exp\on_\rho(\bz) = \sum_{i=0}^\infty Q_i \bz\twisti \in \power{\Mat_n(\C_\infty)}{\bz}
\end{equation}
and
\begin{equation}\label{Logtensor}
\Log _\rho(\bz) = \sum_{i=0}^\infty P_i \bz\twisti \in \power{\Mat_n(\C_\infty)}{\bz}.
\end{equation}



\section{Operators and the Space $\Omega_0$}\label{S:Operators}
In \cite[\S 2.5]{AndThak90} Anderson and Thakur define a $\bA$-module of functions which they call $\Omega_n$ (our notation is $\Omega_0$) which vanish under the operator $\tau - f^n$.  They then connect this space of functions to the period lattice of the exponential function by expressing a function $h\in \Omega_n$ in terms of $t-\theta$, which is a uniformizer at $\Xi$, then analyzing the principle part $h$ in this expansion.  Of particular note, they construct an ancillary vector-valued function $\tilde h$ which they use to aid their calculations in the proof of their period formulas.  In the case of tensor powers of Drinfeld $\bA$-modules we found it necessary to rely entirely upon the equivalent version of $\tilde h$, rather than using it as an ancillary tool.  In this section, we develop a vector setting in which we can embed the space $\Omega_0$ and analyze vector-valued operators on it.

For a fixed a dimension $n$ define
\[\BB := \Gamma \bigl( \cU, \cO_E(-n(V) + n(\Xi))\bigr)\]
where $\cU$ is the inverse image under $t$ of the closed disk in $\C_\infty$ of radius $|\theta|$ centered at 0 defined in \S \ref{S:Background}, and define the $\bA$-module
\begin{equation}\label{Omegadef}
\Omega = \{h\in \BB\mid h\twist - f^n h \in N\},
\end{equation}
where we recall the definition of $N$ from \S \ref{S:Amotives}.  Also define a submodule of $\Omega$ as
\begin{equation}\label{Omega0def}
\Omega_0 = \{h\in \BB\mid h\twist - f^n h = 0\}.
\end{equation}
For a function $h(t,y)\in \Omega$, define the map $T:\Omega \to \TT[y]^n$ by
\begin{equation}\label{Tmapdef}
T( h(t,y)) =
\left (\begin{matrix}
h(t,y)\cdot g_1  \\
h(t,y)\cdot g_2  \\
\vdots  \\
h(t,y)\cdot g_n \\
\end{matrix}\right ),
\end{equation}
where the functions $g_i$ are the basis elements defined in \eqref{gidivisor}.  We observe immediately that $T$ is $\F_q$-linear and injective.

\begin{remark}
The map $T$ can be viewed as a generalization of the $\tilde{\cdot}$ operator defined by Anderson and Thakur in the proof of 2.5.5 of \cite{AndThak90}, where they define for $h(t)\in \TT$,
\[\tilde h(t) =
\left (\begin{matrix}
h(t)\cdot 1  \\
h(t)\cdot (t-\theta)  \\
\vdots  \\
h(t)\cdot (t-\theta)^{n-1}  \\
\end{matrix}\right ).\]
Note that the function $t-\theta$, aside from being a uniformizer at $\Xi$, is also the shtuka function for the Carlitz module, and that it shows up in the $\tau$-basis for the $\bA$-motive associated to the $n$th tensor power of the Carlitz module (see \cite[\S 3.6]{PLogAlg}).  It is not immediately obvious which of these notions leads to the correct generalization of $\tilde{\cdot}$ for Anderson $\bA$-modules.  After noticing properties such as Lemma \ref{L:Operators} and Theorem \ref{T:Period}, however, it seems clear that the definition of $T(\cdot )$ is the correct generalization for the present concerns.
\end{remark}

Define operators on the space $\TT[y]^n$ which act in the sense defined in \S \ref{S:Background} by setting
\[D_t := \rho\on_t - t,\quad\text{and}\quad D_y = \rho\on_y - y.\]
\begin{lemma}\label{L:Operators}
For $h\in \Omega_0$,
\[D_t(T(h)) = D_y(T(h)) = 0.\]
\end{lemma}
\begin{proof}
Using \eqref{P:definingequations} and the fact that $h\in \Omega_0$, observe that
\[t\cdot T(h) = 
\left (\begin{matrix}
th(t,y)\cdot g_1  \\
th(t,y)\cdot g_2  \\
\vdots  \\
th(t,y)\cdot g_n \\
\end{matrix}\right )
=
\left (\begin{matrix}
h(t,y)\cdot (\theta g_1 + a_1 g_2 + g_3)  \\
h(t,y)\cdot (\theta g_2 + a_2 g_3 + g_4)  \\
\vdots  \\
h(t,y)\cdot (\theta g_n + a_n g_1\twist f^n + g_2\twist f^n) \\
\end{matrix}\right )
=
d[\theta] T(h) + E\cdot T(h) \twist.
\]
Thus we see that $\rho\on_t(\tilde h) = t\cdot T(h)$ and so $D_t(T(h)) = 0$.  A similar argument shows that $D_y(T(h)) =~ 0$.
\end{proof}
Define an additional operator on $\TT[y]^n$,
\begin{equation}\label{GE1operator}
G - E_1\tau := 
\left (\begin{matrix}
g_2/g_1 & -1 & 0 & \hdots & 0\\
0 & g_3/g_2 & -1 & \hdots & 0 \\
0 & 0 & g_4/g_3 & \hdots & 0\\
\vdots & \vdots &  \vdots & \ddots & \vdots \\
0 & 0 & 0 & \hdots & g_1\twist f^n/g_n
\end{matrix}\right ) - 
\left (\begin{matrix}
0 & 0 & 0 & \hdots & 0\\
0 & 0 & 0 & \hdots & 0 \\
0 & 0 & 0 & \hdots & 0\\
\vdots & \vdots &  \vdots & \ddots & \vdots \\
1 & 0 & 0 & \hdots & 0
\end{matrix}\right )\tau.
\end{equation}
A quick calculation shows that for any $h\in \Omega_0$
\[\left [G-E_1\tau\right ](T(h)) = 0,\]
and thus the operator $G-E_1\tau$ can be viewed as a vector version of the operator $\tau - f^n$.  In fact, the relationship is even stronger, as proved in the following lemma.

\begin{lemma}\label{L:GE1andOmega}
A vector $J(t,y)\in \TT[y]^n$ satisfies $(G-E_1\tau)(J)=0$ if and only if there exists some function $h(t,y)\in \Omega_0$ such that 
\[J(t,y) = T(h(t,y)).\]
\end{lemma}
\begin{proof}
We have already seen above that $(G-E_1\tau)(T(h))=0$ for all $h\in \Omega_0$.  For the other direction, suppose that $J(t,y)\in \TT[y]^n$ satisfies $(G-E_1\tau)(J)=0$.  Then, if we denote $J = (j_1,\dots, j_n)^\top$, writing out the action of $G-E_1\tau$ on each coordinate gives equations 
\begin{align}\label{eequations}
\nonumber j_1 \frac{g_2}{g_1} - j_2 &= 0\\
\nonumber j_2 \frac{g_3}{g_2} - j_3 &= 0\\
&\vdots\\
\nonumber j_{n-1}\frac{g_n}{g_{n-1}} - j_n &= 0\\
\nonumber j_n\frac{g_1\twist f^n}{g_n} - j_1\twist &= 0.
\end{align}
Solving the first equation for $j_2$ then substituting it into the second, and so on, gives the equality of vectors
\[
\left (\begin{matrix}
j_1  \\
j_2  \\
\vdots  \\
j_n
\end{matrix}\right )
=
\left (\begin{matrix}
j_1  \\
j_1\cdot g_2/g_1  \\
\vdots  \\
j_1\cdot g_n/g_1
\end{matrix}\right ),
\]
and the equality $(\tau - f^n)(j_1/g_1) = 0$, so we see that $J = T(j_1/g_1)$ with $j_1/g_1\in \Omega_0$ as desired.
\end{proof}

We use the quotient functions $g_{k+1}/g_{k}$ frequently throughout this section, so we briefly describe some of their properties.  Using the notation for $k> n$ for $g_k$ from \eqref{gihigher}, the quotients have divisors
\begin{equation}\label{gquotientdiv}
\divisor(g_{k+1}/g_{k}) = (\Xi) - (\infty) + ([k]V\twist + [n - k]V) - ([k-1]V\twist + [n - (k-1)]V),
\end{equation}
for $1\leq k\leq n$.
Thus we can write these functions as a quotient of a linear function of degree 3 and a linear function of degree 2, which we label
\begin{equation}\label{gquotient}
\frac{\nu_k(t,y)}{\delta_k(t)} := \frac{y - \eta - m_k(t-\theta)}{t - t([k-1]V\twist + [n - (k-1)]V)} =\frac{g_{k+1}}{g_k}  ,
\end{equation}
for $1\leq k\leq n$, where $m_k$ is the slope between the points $[k]V\twist + [n - k]V$ and $[-(k-1)]V\twist - [n - (k-1)]V$.
\begin{remark}
The functions $g_{k+1}/g_k$ share many similarities with the shtuka funciton $f$, and they can be viewed as a vector version of the shtuka function; in fact, the divisor of $g_{k+1}/g_k$ matches with the divisor of the shtuka function, except that the points $V\twist$ and $V$ in $\divisor(f)$ from \eqref{fdivisor} are shifted by $([k-1]V\twist+[n-k]V)$.
\end{remark}
With the above analysis we are now equipped to give explicit formulas for the coefficients $a_i$ from Proposition \ref{P:definingequations}, which determine the action of $\rho\on_t$.
\begin{corollary}\label{C:aivalue}
The coefficients $a_i$ from Proposition \ref{P:definingequations} are given by
\[a_i = \frac{2\eta + c_1\theta + c_3}{\theta - t([i]V\twist + [n-i]V)}\]
\end{corollary}
\begin{proof}
Dividing both sides of the first equation from Proposition \ref{P:definingequations} by $g_{i+1}$ and evaluating at the point $-\Xi$ gives
\[a_i = -\frac{g_{i+2}}{g_{i+1}}\bigg|_{-\Xi}.\]
Using expression \eqref{gquotient} for $k=i+1$ we find
\[ -\frac{g_{i+2}}{g_{i+1}}\bigg|_{-\Xi} = \frac{2\eta + c_1\theta + c_3}{\theta - t([i]V\twist + [n-i]V)}.\]
\end{proof}

\begin{remark}
In order to get formulas for $y_i$ and $z_i$ one can equate the coordinates on both sides of the identity
\[\rho\on_{\eta^2 + c_1\eta \theta + c_3\theta} = \rho\on_{\theta ^3 + c_2\theta^2 + c_4\theta + c_6}\]
and solve for the coefficients $y_i$ and $z_i$ in terms of $a_i$.  We do not use this fact going forward, and thus we omit the details.
\end{remark}

Define the operator
\[M_\tau :=N_1 + E_1\tau,\]
where we recall the definition of the matrices $N_i$ and $E_i$ from \eqref{Nidef}.  Denote the diagonal matrix
\begin{equation}\label{Mm}
M_m := \text{diag}(z_1-a_2,z_2-a_3, \dots , z_{n-1} - a_n, z_n - a_1\twist),
\end{equation}
where $a_i$ and $z_i$ are the constants from Proposition \ref{P:definingequations} and denote the diagonal matrix of functions in $H[t,y]$
\begin{equation}\label{Mdelta}
M_\delta:= \text{diag}(\delta_1,\delta_2, \dots , \delta_n).
\end{equation}
\begin{proposition}\label{P:OperatorDecomp}
We have the operator decomposition
\[(G - E_1\tau) = M_\delta\inv (D_y - (M_\tau + M_m) D_t).\]
\end{proposition}
\begin{proof}
We first compute using the definitions \eqref{taction} and \eqref{yaction} and the definitions given above that
\begin{equation}\label{M'def}
D_y - M_\tau D_t - M_m D_t = M',
\end{equation}
where
\[M' := M_1'+M_2'\tau := 
\left (\begin{smallmatrix}
\eta - y - (\theta - t)(z_1 - a_2) & y_1 - (\theta - t) - a_1(z_1 - a_2) & 0 & \hdots & 0\\
0 & \eta - y - (\theta - t)(z_2 - a_3) & y_2 - (\theta - t) - a_2(z_2 - a_3) & \hdots & 0 \\
0 & 0 & \eta - y - (\theta - t)(z_3 - a_4) &  \hdots & 0\\
\vdots & \vdots &  \vdots & \ddots & \vdots \\
0 & 0 & 0 & \hdots & y_{n-1} - (\theta - t) - a_{n-1}(z_{n-1} - a_n)  \\
0 & 0 & 0 & \hdots & \eta - y - (\theta - t)(z_n - a_1\twist)   
\end{smallmatrix}\right )\]
\[
 + \left (\begin{smallmatrix}
0 & 0 & 0 & \hdots & 0\\
0 & 0 & 0 & \hdots & 0 \\
0 & 0 & 0 & \hdots & 0\\
\vdots & \vdots &  \vdots & \ddots & \vdots \\
y_n - (\theta - t) - a_n(z_n - a_1\twist) & 0 & 0 & \hdots & 0
\end{smallmatrix}\right )\tau.\]
If we define $\bg := (g_1, \dots, g_n)^\top$, then by Proposition \ref{P:definingequations} we observe that
\[d[\theta]\bg + E_\theta f^n\bg\twist = 0 \quad \text{and} \quad d[\eta]\bg + E_\eta f^n\bg\twist.\]
Then from \eqref{M'def}, we observe that
\begin{equation}\label{bgidentity}
M_1'\bg + M_2' f^n \bg\twist = 0.
\end{equation}
Examining the coordinates of the above equation gives the equations for $1\leq k\leq n-1$
\begin{equation}\label{gquotientexpression}
\frac{g_{k+1}}{g_k} = \frac{y-\eta - (\theta - t)(z_k-a_{k+1})}{t - \theta + y_k - a_k(z_k - a_{k+1})},
\end{equation}
and
\[\frac{f^n g_{1}\twist}{g_n} = \frac{y-\eta - (\theta - t)(z_n-a_{1}\twist)}{t - \theta + y_n - a_n(z_n - a_1\twist)}.\]
Comparing these formulas with the notation established in \eqref{gquotient} shows that for $1\leq k\leq n-1$
\[m_k = z_k - a_{k+1} \quad \text{and} \quad \delta_k = t - \theta + y_k - a_km_k\]
and
\[m_n = z_n - a_1\twist \quad \text{and} \quad \delta_n = t - \theta + y_n - a_nm_n.\]
With these observations, we then identify
\[M' = M_\delta(G - E_1\tau),\]
so that
\[(G - E_1\tau) = M_\delta\inv (D_y - (M_\tau + M_m) D_t).\]
\end{proof}

\begin{remark}
Note the similarity of this decomposition to that in \cite[Prop. 4.1]{GP16}.
\end{remark}

Althought we do not use the following corollary in the present paper, it is used in the concurrent paper \cite{GreenZeta17}, and so we record it here.

\begin{corollary}\label{C:OperatorDecomposition2}
Define the following matrices
\[M_1 = M_1'\big|_{t=0,y=0}\quad \text{and}\quad M_2 = M_2'\big|_{t=0,y=0},\]
with $M_1'$ and $M_2'$ as in \eqref{M'def}.  Then
\[\rho\on_y - (M_\tau + M_m) \rho\on_t = M_1 + M_2\tau.\]
\end{corollary}
\begin{proof}
After multiplying both sides by $M_\delta$, the matrices in Proposition \ref{P:OperatorDecomp} have coefficients in $\overline K[t,y]$, and equating the constant terms gives the corollary.
\end{proof}

Define the function
\begin{equation} \label{omegarhoprod}
  \omega_\rho = \xi^{1/(q-1)} \prod_{i=0}^\infty \frac{\xi^{q^i}}{f^{(i)}}, \quad \xi = -\frac{m\theta -\eta}{\alpha} = -\biggl( m + \frac{\beta}{\alpha} \biggr),
\end{equation}
where $m$, $\alpha$, and $\beta$ are given in \S \ref{S:Amotives} and recall that $\omega_\rho \in \TT[y]^\times$ (see \cite[\S 4]{GP16}, for details of convergence).  Note that
\begin{equation}\label{omegarhoeq}
(\omega_\rho^n)\twist = f^n\omega_\rho^n,
\end{equation}
and thus $\omega_\rho^n \in \Omega_0$.  The function $\omega_\rho$ was originally defined for tensor powers of the Carlitz module by Anderson and Thakur in \cite[\S 2.5]{AndThak90} and was generalized to Drinfeld modules by the Papanikolas and the author in \cite{GP16}.
\begin{proposition} \label{P:Omegaprop}
The function $\omega_\rho^n$ generates $\Omega_0$ as a free $\bA$-module.
\end{proposition}
\begin{proof}
The proof follows very similarly to the proof of \cite[Prop. 4.3]{GP16}.  One observes that since $\omega_\rho^n\in \TT[y]^\times$, for any $h\in \Omega_0$ the quotient $h/\omega_\rho^n$ is fixed under twisting and thus is in $\bA$.  We leave the details to the reader.
\end{proof}

\section{Anderson Generating Functions and Periods}\label{S:Period}
Anderson and Thakur studied the period lattice of the $n$-fold tensor power of the Carlitz module in \cite{AndThak90}, where they find succinct formulas for the last coordinate of a fundamental period.  On the other hand, Gekeler, Goss, Thakur, Papanikolas and Lutes and Papanikolas and Chang have studied the fundamental period associated to (1-dimensional) Drinfeld modules (see \cite[\S III]{Gekeler}, \cite[\S 7.10]{Goss}, \cite[Ex.~4.15]{LP13}, Thakur~\cite[\S 3]{Thakur91} and \cite{CP11}-\cite{CP12} respectively).  More recently, Papanikolas and the author studied periods of rank 1 sign-normalized Drinfeld modules in \cite{GP16} using Anderson generating functions.  This section generalizes the work of both Anderson and Thakur and of Papanikolas and the author; we develop the theory of periods of $n$-fold tensor powers of rank 1 sign-normalized Drinfeld modules.  We remark that because of the additional complexity arising from generalizing in both these directions, our methods required several new ideas, distinct from the works mentioned above.  In particular, while the residue formula presented in Proposition \ref{P:Residue} is nearly trivial in the 1-dimensional case, its proof for the $n$-dimensional case required several new technical insights to account for the higher order poles present in vector-valued Anderson generating.

We now define and study vector-valued Andreson generating functions in dimension $n$.  Such functions are used in the proof of Theorem 2.5.5 in \cite{AndThak90} for the case of tensor powers of the Carlitz module; here we define them for Anderson $\bA$-modules.  For $\bu = (u_1,...,u_n)^\top \in \C_\infty^n$ define
\begin{equation}\label{Eudef}
E_{\bu}^{\otimes n}(t) := \left ( \begin{matrix}
e_1(t)\\
\vdots\\
e_n(t)
\end{matrix}\right) :=  \sum_{i=0}^\infty \Exp_\rho\on\left (d[\theta]^{-i-1} \bu\right ) t^i,
\end{equation}
then define
\begin{equation}\label{Gudef}
G_\bu\on (t,y) := E_{d[\eta]\bu}\on(t) + (y+c_1t + c_3) E_\bu\on(t).
\end{equation}

We will shortly discuss the convergence of $E_\bu\on$ and $G_\bu\on$ as functions in Tate algebras, but before proceeding we require two brief lemmas.

\begin{lemma}\label{L:MatrixGeoseries}
Given an upper triangular matrix $M\in \Mat_n(\TT)$ with eigenvalues $\lambda_i \in \TT$, the series 
\[\sum_{i=0}^\infty M^i\]
converges with respect to $\lVert \cdot\rVert$ and equals $(I-M)\inv$ if and only if $|\lambda_i|<1$ for all $1\leq i\leq n$.
\end{lemma}
\begin{proof}
This is essentially a standard result from linear algebra, so we only sketch the proof.  We write $M = D + N$ where $D$ is the diagonal matrix consisting of eigenvalues and $N$ is a strictly upper triangular matrix.  Then we write $M^i =(D+N)^i$ and expand $(D+N)^i$ to find that any term with $n$ or more copies of $N$ vanishes.  Thus $\lVert M^i \rVert\to 0$ as $i\to 0$ if and only if $|\lambda_i|<1$.
\end{proof}

\begin{lemma}\label{L:Coordregular}
The coordinates of the matrix
\[(d[\eta] - y)\left (d[\theta] -   t \right )\inv,\]
are regular at $\Xi$, where $d:\bA \to \Mat_n(H)$ is the ring homomorphism from \S \ref{S:Amodules}.
\end{lemma}
\begin{proof}
For ease of exposition in this proof we will assume that the elliptic curve $E$ has the simplified Weierstrass equation $E:\, y^2 = t^3 + At + B$ for $A,B\in \F_q$.  The lemma holds for the more general Weierstrass equation \eqref{ecequation} and we leave the extra details to the reader.  Observe using the simplified Weierstrass equation together with the fact that $d:\bA \to \Mat_n(H)$ is a ring homomorphism that
\begin{align*}
(d[\eta] - y)\left (d[\theta] -   t \right )\inv &= (d[\eta] - y)(d[\eta]+y)(d[\eta]+y)\inv\left (d[\theta] -   t \right )\inv \\
&= (d[\eta^2] -y^2)(d[\eta]+y)\inv\left (d[\theta] -   t \right )\inv\\
&= ((d[\theta^3]-t^3) +A(d[\theta]-t) )(d[\eta]+y)\inv\left (d[\theta] -   t \right )\inv\\
& = ((d[\theta^2] + td[\theta]-t^2) + A)(d[\eta]+y)\inv,
\end{align*}
where in the last equality we factored out $(d[\theta] - t)$ and canceled.  Note that $(d[\eta]+y)\inv$ and $(d[\theta^2] + td[\theta]-t^2) + A$ are coordinate-wise regular at $\Xi$ and thus so is $(d[\eta] - y)\left (d[\theta] -   t \right )\inv$.  
\end{proof}

For the case of $n=1$ and $A=\F_q[\theta]$, El-Guindy and Papanikolas show that Anderson generating functions are in $\TT$ and prove they have a meromorphic continuation to $\C_\infty$ in \cite{EP14}.  We give a similar theorem for $E_\bu\on$ and $G_\bu\on$.

\begin{proposition}\label{P:Euspace}
For $\bu\in \C_\infty^n$, the function $E_\bu\on \in \TT^n$ and we have the following identity of functions in $\TT^n$
\[E_\bu\on(t) = \sum_{j=0}^\infty Q_j\left (d[\theta]\twistk{j} -   tI \right )\inv\bu\twistk{j},\]
where $Q_i$ are the coefficients of $\Exp_\rho\on$ from \eqref{Exptensor}. Further, the function $G_\bu\on$ extends to a meromorphic function on  $U = (\C_\infty \times_{\F_q} E) \setminus \{ \infty \}$ with poles in each coordinate only at the points $\Xi\twisti$ for $i\geq 0$.
\end{proposition}
\begin{proof}
Writing in the definition of $\Exp_\rho\on$ from \eqref{Exptensor} and expanding gives the sum
\begin{equation}\label{Eudoublesum}
E_\bu\on(t) = \sum_{i=0}^\infty \left (\sum_{j=0}^\infty Q_j\left (d[\theta]^{-i-1} \bu \right )\twistk{j} \right ) t^i.
\end{equation}
Recall from \eqref{Nthetadef} that $d[\theta] = \theta I + N_\theta$ where $N_\theta$ is nilpotent with order $n$, so we can write
\begin{align}\label{degreeest1}
\begin{split}
\left (d[\theta]^{-i-1}  \right ) &= \left ((\theta I + N_\theta)^{-i-1}  \right )\\
&= \left (\left (\frac{1}{\theta}I - \frac{1}{\theta^2} N_\theta + \dots + \frac{(-1)^{n-1}}{\theta^n}N_\theta^{n-1}\right )^{i+1}  \right )\\
&= \left (\left [ \sum_{k_1+\dots+ k_n = i+1} \binom{i+1}{k_1,\dots,k_n}\prod_{s=1}^n \left ( \frac{1}{\theta^s} N_\theta^{s-1}\right )^{k_s}  \right ]  \right )\\
&= \left (\left (\frac{1}{\theta^{i+1}}I + d_1 \frac{1}{\theta^{i+2}} N_\theta + \dots + d_{n-1}\frac{1}{\theta^{i+n}}N_\theta^{n-1}\right ) \right ),
\end{split}
\end{align}
where in the last two line we used the multinomial theorem then collected like terms using some constants $d_i\in \F_q$.  Using the last line of \eqref{degreeest1} we find that
\begin{equation}\label{coefinequality}
\left |\sum_{j=0}^\infty Q_j\left (d[\theta]^{-i-1} \bu \right )\twistk{j}\right | \leq \max_{j}\left \{\lvert Q_j\rvert \cdot |\theta|^{-iq^j}\max_{1\leq k\leq n}\left \{\left | \frac{1}{\theta^{k}} N_\theta^{k-1}\right |\right \}^{q^j}\cdot |\bu|^{q^j}\right \},
\end{equation}
where $|\cdot|$ is the matrix seminorm defined in \S \ref{S:Background}.  Let us denote
\[N_0 = \max_{1\leq k\leq n}\left \{\left | \frac{1}{\theta^{k}} N_\theta^{k-1}\right |\right \},\]
which equals some constant independent of $i$ and $j$.  Then, the fact that $\Exp_\rho\on$ is an entire function on $\C_\infty^n$, implies that the factor
\[\lvert Q_j\rvert \cdot \max_{1\leq k\leq n}\left \{\left | \frac{1}{\theta^{k}} N_\theta^{k-1}\right |\right \}^{q^j}\cdot |\bu|^{q^j}\]
goes to zero as $j\to \infty$, and thus is bounded independent of $j$.  Thus by \eqref{coefinequality}
\[\left |\sum_{j=0}^\infty Q_j\left (d[\theta]^{-i-1} \bu \right )\twistk{j}\right |\]
goes to zero as $i\to \infty$, which proves that $E_\bu\on\in \TT^n$.  Further, using the above analysis, we find that
\[\left |Q_j\left (d[\theta]^{-i-1} \bu \right )\twistk{j}\right |\to 0,\]
as $\max(i,j)\to 0$, and thus we are allowed to rearrange the terms of the double sum \eqref{Eudoublesum} and maintain convergence in $\TT^n$ (see \cite[\S 1.2]{Robert}).

Next, observe that the eigenvalues of the matrix $d[\theta]\inv t$ are all equal to $t/\theta$, and that $\lVert t/\theta\rVert <1$, and hence by Lemma \ref{L:MatrixGeoseries} we have the geometric series identity in $\TT^n$
\[\sum_{i=0}^\infty d[\theta]^{-i-1}  t^i = \left (d[\theta]\twistk{j} -   tI \right )\inv.\]
Using this we rearrange the terms of $E_\bu\on$ to get the equality in $\TT^n$
\begin{align*}
E_\bu\on &= \sum_{j=0}^\infty Q_j\left (\sum_{i=0}^\infty d[\theta]^{-i-1}  t^i \right )\twistk{j}\bu\twistk{j}\\
&= \sum_{j=0}^\infty Q_j\left (d[\theta]\twistk{j} -   tI \right )\inv\bu\twistk{j}.
\end{align*}
Using the above equation, we see that
\begin{equation}\label{Gugeoseries}
G_\bu\on=  \sum_{j=0}^\infty Q_j\left (d[\theta]\twistk{j} -   tI \right )\inv\left (d[\eta]\twistk{j} + (y+c_1t+c_3)I\right )\bu\twistk{j} \in \TT[y]^n.
\end{equation}
We then observe, using analysis similar to that in \eqref{degreeest1}, that for any $m\geq 0$ the sum
\begin{align*}
\sum_{j=0}^\infty Q_j&\left (d[\theta]\twistk{j} -   tI \right )\inv\left (d[\eta]\twistk{j} + (y+c_1t+c_3)I\right )\bu\twistk{j} \\
&- \sum_{j=0}^m Q_j\left (d[\theta]\twistk{j} -   tI \right )\inv\left (d[\eta]\twistk{j} + (y+c_1t+c_3)I\right )\bu\twistk{j}
\end{align*}
converges for any point $(t,y)\in U$ with $|t|<|\theta|^{m+1}$, providing a meromorphic continuation of $G_\bu\on$ to $U$.  We also observe that the only possible poles in each coordinate of
\begin{equation}\label{Eupartialsum}
\sum_{j=0}^m Q_j\left (d[\theta]\twistk{j} -   tI \right )\inv\left (d[\eta]\twistk{j} + (y+c_1t+c_3)I\right )\bu\twistk{j} \in H(t,y)^n
\end{equation}
occur at $\pm\Xi\twisti$ for $i\leq m$.  We calculate that each coordinate of $G_\bu\on$ does actually have poles at the positive twists of $\Xi$ (see the proof of Proposition \ref{P:Residue} for more details).  On the other hand, under the substitution given by negation on $E$, namely $(t,y)\mapsto (t,-y-c_1t-c_3)$ we see that
\[\left (d[\theta]\twistk{j} -   tI \right )\left (d[\eta]\twistk{j} + (y+c_1t+c_3)I\right ) \mapsto \left (d[\theta]\twistk{j} -   tI \right )\left (d[\eta]\twistk{j} - yI\right ),\]
and so by Lemma \ref{L:Coordregular} we see that each coordinate of \eqref{Eupartialsum} is regular at $-\Xi\twistk{j}$ for $j\geq 0$.   Thus the meromorphic continuation described above has the correct properties.
\end{proof}

\begin{lemma}\label{L:DtofGu}
For $\bu\in \C_\infty^n$, we obtain two identities
\begin{align*}
\mathrm{(a)}\,\, D_t(G_\bu\on) &= \Exp_\rho\on(d[\eta]\bu) + (y+c_1t + c_3)\Exp_\rho\on(\bu)\\
\mathrm{(b)}\,\,D_y(G_\bu\on) &= -c_1 \Exp_\rho\on(d[\eta] \bu) + \Exp_\rho\on(d[\theta^2] \bu)+ (t+c_2)\Exp_\rho\on(d[\theta] \bu)\\
&  + (t^2 + c_2 t + c_4) \Exp_\rho\on(\bu).
\end{align*}
\end{lemma}
\begin{proof}
First observe that 
\[
\rho\on_t(E_\bu\on) = \sum_{i=0}^\infty \rho\on_t\left (\Exp_\rho\on\left (d[\theta]^{-i-1} \bu\right )\right ) t^i = \Exp_\rho\on(\bu) + t E_\bu\on,
\]
and thus
\[D_t(E_\bu\on) = \Exp_\rho\on(\bu).\]
Part (a) of the lemma follows directly from this.  For part (b), observe that
\[\rho\on_y(E_{\bu}\on) = \sum_{i=0}^\infty\left (\Exp_\rho\on\left (d[\eta]d[\theta]^{-i-1} \bu\right )\right ) t^i,\]
and so using \eqref{ecequation}
\begin{align*}
\rho\on_y(E_{d[\eta]\bu}\on) &= \sum_{i=0}^\infty\left (\Exp_\rho\on\left (d[\eta^2]d[\theta]^{-i-1} \bu\right )\right ) t^i\\
&= \sum_{i=0}^\infty\left (\Exp_\rho\on\left (d[\theta^3 + c_2\theta^2 + c_4 \theta + c_6 -c_1 \theta\eta - c_3 \eta ]d[\theta]^{-i-1} \bu\right )\right ) t^i.
\end{align*}
Using the above equation, and the $\F_q$-linearity of $\Exp_\rho\on$, one examines the terms of $D_y(G_\bu\on)$ and finds that all but a finite number cancel and the remaining terms are exactly the right hand side of part (b) of the lemma.  We leave the details to the reader. 
\end{proof}

Define $\cM$ to be the subring of $\TT[y]$ consisting of all elements in $\TT[y]$ which have a meromorphic continuation to all of $U$.  Now define the map
\[\RES_{\Xi}:\cM^n\to \C_\infty^n,\]
for a vector of functions $(z_1(t,y),...,z_n(t,y))^\top\in \cM^n$ as
\begin{equation}\label{RESdef}
\RES_{\Xi}
\left (\begin{matrix}
z_1(t,y)  \\
\vdots  \\
z_n(t,y) \\
\end{matrix}\right )
= 
\left (\begin{matrix}
\Res_{\Xi}(z_1(t,y)\lambda)  \\
\vdots  \\
\Res_{\Xi}(z_n(t,y)\lambda) \\
\end{matrix}\right ),
\end{equation}
where $\lambda$ is the invariant differential of $E$ from \eqref{invdiff}.  We remark that in defining the maps $T$ and $\RES_{\Xi\twisti}$, we were partially inspired by ideas of Sinha in \cite[\S 4.6.6]{Sinha97}.  We now analyze the residues of the Anderson generating function $G_\bu\on$ under the map $\RES_\Xi$.

\begin{proposition}\label{P:Residue}
If we write $\bu = (u_1,...,u_n)^\top\in \C_\infty^n$, then
\[\RES_\Xi(G_\bu\on) = -(u_1,...,u_n)^\top.\]
\end{proposition}
\begin{proof}
Again, for ease of exposition in this proof we will assume that the elliptic curve $E$ has the simplified Weierstrass equation $E:\, y^2 = t^3 + At + B$ for $A,B\in \F_q$.  The proposition holds for the more general Weierstrass equation \eqref{ecequation} and we leave the extra details to the reader.  Equation \eqref{Gugeoseries} gives
\[G_\bu\on=  \sum_{j=0}^\infty Q_j\left (d[\theta]\twistk{j} -   tI \right )\inv\left (d[\eta]\twistk{j} +y I\right )\bu\twistk{j},\]
so when we calculate $\RES_\Xi(G_\bu\on\lambda)$, we find that the only possible contributions to the residues come from the $j=0$ term, since $\left (d[\theta]\twistk{j} -   tI \right )\inv$ is regular at $\Xi$ in each coordinate for $j\geq 1$.  In particular, we find that
\[\RES_\Xi(G_\bu\on) = \RES_\Xi\left (\left (d[\eta] + yI\right )\left (d[\theta] -   tI \right )\inv\bu\right  ),\]
and further that
\begin{align}\label{RESequation1}
\begin{split}
\left (d[\eta] + yI\right )\left (d[\theta] -   tI \right )\inv\lambda &=\left (d[\eta] + yI\right )\left (d[\theta] -   tI \right )\inv\cdot \frac{dt}{2y}\\
&=\frac{1}{2}\left (2d[\eta] -( d[\eta] - y) \right )(d[\theta] -   tI  )\inv \left (\frac{1}{y} - d[\eta]\inv + d[\eta]\inv \right )dt\\
&=\frac{1}{2}\left (2d[\eta] -( d[\eta] - y) \right )(d[\theta] -   tI  )\inv  \left (\frac{d[\eta]}{y}\inv(d[\eta]-yI) + d[\eta]\inv \right )dt\\
\end{split}
\end{align}
After multiplying out the factors in the last line of \eqref{RESequation1}, using Lemma \ref{L:Coordregular} we find that the only term whose coordinates have poles at $\Xi$ is $(d[\theta]-t)\inv$.  Thus we see that
\[\left (d[\eta] + yI\right )\left (d[\theta] -   tI \right )\inv\lambda = (d[\theta]-t)\inv dt + \mathbf{r}(t,y)dt,\]
where $\mathbf r(t,y)\in H(t,y)^n$ is some function which is regular at $\Xi$ in each coordinate.  Recall the definition of the matrix
\[d[\theta] -   tI = 
\left (\begin{matrix}
(\theta - t) & a_1 & 1 & \hdots & 0\\
0 & (\theta - t) & a_2 & \hdots & 0 \\
0 & 0 & (\theta - t) & \hdots & 0\\
\vdots & \vdots &  \vdots & \ddots & \vdots \\
0 & 0 & 0 & \hdots & (\theta - t)
\end{matrix}\right ),\]
where the constants $a_i\in H$ are from Proposition \ref{P:definingequations}.  Because the matrix is upper triangular, we see immediately that the inverse matrix has the form
\[\left (d[\theta] -   tI \right )\inv =
\left (\begin{matrix}
\frac{1}{\theta - t} & * & * & \hdots & *\\
0 & \frac{1}{\theta - t} & * & \hdots & * \\
0 & 0 & \frac{1}{\theta - t} & \hdots & *\\
\vdots & \vdots & \vdots & \ddots & \vdots \\
0 & 0 & 0 & \hdots & \frac{1}{\theta - t}
\end{matrix}\right ),\]
where each off diagonal entry denoted by $*$ is a sum of the form
\[\sum_{k=k_0}^{n}\frac{d_k}{(t-\theta)^k}\]
for $k_0\geq 0$ and some (possibly zero) constants $d_k\in H$.  Using the cofactor expansion of the inverse, we find that $k_0\geq 2$ for each coordinate, and thus the off diagonal entries will not contribute to the residue.  Thus, since $t-\theta$ is a uniformizer at $\Xi$, for some functions $r_i(t)\in H(t,y)$ which have no residue at $\Xi$ we find that
\begin{equation}\label{RESfinaleq}
\RES_\Xi(G_\bu\on) = 
\left (\begin{matrix}
\Res_\Xi\left ((\frac{u_1}{\theta-t} + r_1(t))dt\right )\\
\vdots\\
\Res_\Xi\left ((\frac{u_n}{\theta-t} + r_n(t))dt\right )
\end{matrix}\right )
=
-\left (\begin{matrix}
u_1\\
\vdots\\
u_n
\end{matrix}\right ).
\end{equation}

\end{proof}

\begin{proposition}\label{P:CompositionMaps}
The composition of maps
\[\RES_\Xi\circ\, T: \Omega_0\to \C_\infty^n\]
is an injective $\bA$-module homomorphism, where $\bA$ acts on $\Omega_0$ by multiplication and on $\C_\infty^n$ by $\rho\on$, and its image is $\lambda_\rho\on = \ker(\Exp_\rho\on )$.
\end{proposition}
\begin{proof}
The proof follows similarly to the proof of \cite[Thm. 4.5]{GP16}.  For an arbitrary $h\in \Omega^n$, each coordinate of $T(h)$ is in $\TT[y]$, so we can write
\[T(h) = \sum_{i=0}^\infty \bb_{i+1} t^i + (y + c_1t + c_3) \sum_{i=0}^\infty \bc_{i+1} t^i\]
uniquely for $\bb_i,\bc_i\in \C_\infty^n$.  Then using Lemma \ref{L:Operators}, we observe that
\begin{align*}
\sum_{i=0}^\infty \rho\on_t (\bb_{i+1})t^i + (y+c_1t+c_3)\sum_{i=0}^\infty \rho\on_t (\bc_{i+1})t^i &= \rho\on_t (T(h)) \\
& = tT(h)\\
& = \sum_{i=0}^\infty \bb_{i+1} t^{i+1} + (y + c_1t + c_3) \sum_{i=0}^\infty \bc_{i+1} t^{i+1},
\end{align*}
from which we see that if we set $\bb_0=\bc_0 =0$, then for $i\geq 0$
\begin{equation}\label{rhotcoef}
\rho\on_t (\bb_{i+1}) = \bb_{i}, \quad \rho\on_t (\bc_{i+1}) = \bc_{i}.
\end{equation}
Similarly we find that for $i\geq 0$
\begin{equation}\label{rhoycoef}
\rho\on_y (\bc_{i}) = \bb_i.
\end{equation}
Since $|\bb_i|$, $|\bc_i|\to 0$ as $i\to \infty$, there is some $i_0>0$ such that $\bb_{i+1}$ and $\bc_{i+1}$ both lie within the radius of convergence of $\Log_\rho\on$ for $i>i_0$.  Thus by \eqref{logfunctionalequation} and \eqref{rhotcoef}, for $i>i_0$ we have
\[
  d[\theta^i] \Log_\rho\on(\bb_i) = d[\theta^{i+1}]\Log_\rho\on(\bb_{i+1}), \quad d[\theta^i] \Log_\rho\on(\bc_i) = d[\theta^{i+1}]\Log_\rho\on(\bc_{i+1}),
\]
and we note that these two quantities are independent of $i$.  We set
\[
\Pi_n := d[\theta^i]\Log_\rho\on(\bc_i),
\]
for some $i > i_0$, and note that
\[
d[\eta]\Pi_n = d[\eta] d[\theta^i]\Log_\rho\on(\bc_i) = d[ \theta^i]d[\eta]\Log_\rho\on(\bc_i) = d[\theta^i]\Log_\rho\on(\rho\on_y(\bc_i)) = d[\theta^i]\Log_\rho\on(\bb_i).
\]
Using \eqref{expfunctionalequation} together with the above discussion we see that
\[
\Exp_\rho\on (\Pi_n) = \Exp_\rho\on(d[\theta^i]\Log_\rho\on(\bc_i)) = \rho\on_{t^i} (\bc_i) = \rho\on_t(\bc_1) =\bc_0= 0,
\]
which implies that $\Pi_n \in \lambda_\rho\on$.  Further, we see that
\[
\bb_i = \Exp_\rho\on\left (d[\eta]d[\theta^{-i}] \Pi_n\right ), \quad \bc_i = \Exp_\rho\on\left ( d[\theta^{-i}] \Pi_n\right ),
\]
and thus
\[
T(h) = G_{\Pi_n}\on = E_{d[\eta] \Pi_n}\on + (y+c_1t + c_3)E_{\Pi_n}\on.
\]
By Proposition~\ref{P:Residue}, we see that $\RES_\Xi(T(h)) = -\Pi_n$, and thus $\RES_\Xi(T(\Omega_0)) \subseteq \lambda_\rho\on$.  Since $G_{\Pi_n}\on = G_{\Pi_n'}\on$ if and only if $\Pi_n = \Pi_n'$, the map $\RES_\Xi\circ\, T$ is injective.  Finally, let $\Pi_n'\in \lambda_\rho\on$, so that Lemma \ref{L:DtofGu} shows that
\[D_t(G_{\Pi_n'}\on)=D_y(G_{\Pi_n'}\on) = 0.\]
Thus, using Proposition \ref{P:OperatorDecomp} we find that
\[(G-E_1\tau)(G_{\Pi_n'}\on) = 0,\]
and hence by Lemma \ref{L:GE1andOmega} $G_{\Pi_n'}\on = T(h)$ for some function $h\in \Omega_0$.  Finally, by Proposition \ref{P:Residue}
\[\RES_\Xi(T(h)) =\RES_\Xi(G_{\Pi_n'}\on) = \Pi_n'\]
which shows that $\lambda_\rho\on \subset \RES_\Xi(T(\Omega_0))$.  To see that $\RES_\Xi\circ T$ is an $\bA$-module homomorphism, for $h\in \Omega_0$, using the above discussion we find that
\[\RES_\Xi(T(th)) = \RES_\Xi(tG_{\Pi_n'}\on),\]
for some $\Pi_n'\in \Lambda_\rho\on$ and using analysis similar to that in the proof of Proposition \ref{P:Residue} that
\[\RES_\Xi(tG_{\Pi_n'}\on) = \RES_\Xi((t - d[\theta])G_{\Pi_n'}\on + d[\theta]G_{\Pi_n'}\on) = d[\theta]\RES_\Xi(G_{\Pi_n'}\on)= d[\theta]\RES_\Xi(T(h)).\]
It follows similarly that $\RES_\Xi(T(yh)) = d[\eta]\RES_\Xi(T(h))$, which finishes the proof.
\end{proof}
\begin{theorem}\label{T:Period}
If we denote
\[\Pi_n = -\RES_\Xi(T(\omega_\rho^n)),\]
then $T(\omega_\rho^n) = G_{\Pi_n}\on$ and $\lambda_\rho\on = \{d[a]\Pi_n\mid a\in \bA\}$.  Further, if $\pi_\rho$ is a fundamental period of the (1-dimensional) Drinfeld exponential function $\exp_\rho$ from \eqref{1dimperiod}, then the last coordinate of $\Pi_n \in \C_\infty^n$ is
\[ \frac{g_1(\Xi)}{a_1a_2\dots a_{n-1}}\cdot\pi_\rho^n,\]
where the constants $a_i$ are from Proposition \ref{P:definingequations}.
\end{theorem}
\begin{proof}
The first two statements follow immediately from Propositions \ref{P:Omegaprop} and \ref{P:CompositionMaps}.  Then recall from \cite{GP16} that $ \pi_\rho  = -\!\Res_{\Xi}(\omega_\rho \lambda)$, whereupon the last statement follows by noting that the last coordinate of $-\RES_\Xi(T(\omega_\rho^n))$ equals
\[-\Res_\Xi(\omega_\rho^n g_n\lambda) = -\Res_\Xi\left ((t-\theta)^{n-1}\omega_\rho^n\lambda\right )\cdot \left (\frac{g_n}{(t-\theta)^{n-1}}\bigg |_{\Xi}\right ) = \pi_\rho^n \cdot \left (\frac{g_n}{(t-\theta)^{n-1}}\bigg |_{\Xi}\right ),\]
since $(t-\theta)^{n-1}\omega_\rho^n$ has a simple pole at $\Xi$ and since $g_n/(t-\theta)^{n-1}$ is regular at $\Xi$.  The formula then follows by dividing the first equation of Proposition \ref{P:definingequations} through by $g_{i+1}$ then evaluating at $\Xi$ to get
\[\frac{(t-\theta)g_i}{g_{i+1}}\bigg|_\Xi = a_i.\]
\end{proof}

\section{Example}\label{S:examples}
\begin{example}\label{E:easyexample}
Let $E:y^2 = t^3-t-1$ be defined over $\F_3$.  Then from \cite{Thakur93} we find that
\[f = \frac{y-\eta - \eta(t-\theta)}{t - \theta - 1}.\]
We form the $2$-dimensional Anderson $\bA$-module as outlined in section \S \ref{S:Amodules}, where we recall from \eqref{gidivisor} that
\[\divisor(g_1) = -2(V) + (\infty) + ([2]V),\quad \divisor(g_2) = -2(V) + (\Xi) + (V\twist + V).\]
If we denote $T_{-V}$ as translation by $-V$ on $E$, then we can quickly write down formulas for $g_1$ and $g_2$ by observing that $g_1\circ T_{-V}$ and $g_2\circ T_{-V}$ are both polynomials with relatively simple divisors, from which we calculate that
\[g_1 = \frac{\eta^{2} + \eta y + t - \theta - 1}{\eta t^{2} + \eta t \theta + \eta \theta^{2} + \eta t - \eta \theta + \eta},\]
\[g_2 =  \frac{\eta^{2} t^{2} + \eta^{2} t \theta + \eta^{2} \theta^{2} + \eta^{2} t - \eta^{2} \theta - \eta^{2} + t^{2} + t \theta + \theta^{2} + \eta y - t + \theta}{\eta^{2} t^{2} + \eta^{2} t \theta + \eta^{2} \theta^{2} + \eta^{2} t - \eta^{2} \theta + \eta^{2} + t^{2} + t \theta + \theta^{2} + t - \theta + 1}.\]
Then using Corollary \ref{C:aivalue} we calculate that
\[\rho\on_t =
\left (\begin{matrix}
\theta & \frac{-(\eta^2+1)^2}{\eta^3}  \\
 0 &\theta
\end{matrix}\right )+
\left (\begin{matrix}
1 & 0  \\
\frac{-\eta^3(\eta^4-\eta^2 - 1)}{(\eta^2+1)^3} & 1
\end{matrix}\right )\cdot \tau.\]
We then calculate that the bottom coordinate of $\Pi_2$ from Theorem \ref{T:Period} is
\[\frac{-(\eta^2 + 1)^2}{(\eta^5 - \eta^3 - \eta)} \cdot \pi_\rho^2.\]
\end{example}

\end{document}